\documentclass[12pt, titlepage]{article}

\usepackage{eurosym}
\usepackage[dvipsnames]{xcolor}
\usepackage{amsmath,amssymb,amsthm,amsfonts}
\usepackage[mathscr]{eucal}
\usepackage{fullpage}
\usepackage{graphicx}
\usepackage{subfig}
\usepackage{epstopdf}
\usepackage{lscape}
\usepackage{setspace}
\usepackage{epsfig}
\usepackage[round]{natbib}
\usepackage{multirow}
\usepackage{color}
\usepackage[latin1]{inputenc}
\usepackage{bbm}
\usepackage{mathtools}
\usepackage{appendix}

\numberwithin{equation}{section} 

\newtheorem{theorem}{Theorem}[section]
\newtheorem{corollary}{Corollary}[section]

\newtheorem{proposition}{Proposition}[section]
\newtheorem{definition}{Definition}[section]
\newtheorem{remark}{Remark}

\newcommand{\ba}{\begin{eqnarray}}
\newcommand{\ea}{\end{eqnarray}}
\def\XX{\boldsymbol{X}}
\def\xx{\boldsymbol{x}}

\def\GG{\boldsymbol{G}}

\def\AA{\boldsymbol{A}}

\def\SS{\boldsymbol{S}}
\def\ss{\boldsymbol{s}}

\def\w{\boldsymbol{w}}

\def\aa{\boldsymbol{a}}

\def\YY{\boldsymbol{Y}}

\def\BB{\boldsymbol{B}}
\def\ets{\mathcal{E}T\alpha S}
\def\II{\boldsymbol{I}}

\def\dd{\boldsymbol{\delta}}
\def\mmu{\boldsymbol{\mu}}
\def\hmu{\hat{\mu}}

\def\RR{\mathbb{R}}

\def\zz{\boldsymbol{z}}

\def\ee{\boldsymbol{e}}

\begin{document}

\title{\bf Multivariate tempered stable   additive subordination for financial models}

\author{
{\bf Patrizia Semeraro}  \\
       Department of Mathematical Science\\
       Politecnico di Torino\\
       c.so Duca degli Abruzzi, 24 \\
       Torino, Italy\\[4pt]}
\date{}
\maketitle

\begin{abstract}

We  study a class of multivariate tempered stable distributions  and introduce the associated class of tempered stable Sato  subordinators.
These  Sato subordinators are  used to build additive inhomogeneous processes   by subordination of a multiparameter Brownian motion. The resulting process is additive and time inhomogeneous.   Furthermore, these processes are associated with the distribution at  unit time of a class of L\'evy process with good fit properties on financial data.
The main feature of the Sato subordinated Brownian motion is that it has time dependent correlation, whereas the L\'evy counterpart does not. We choose a  specification with normal inverse Gaussian  distribution at unit time and provide a numerical illustration of the correlation dynamics.

 \vfil


\end{abstract}

\section*{Introduction}

Additive processes with independent but inhomogeneous increments have been proposed to model asset returns.
 \cite{carr2007self} showed that these processes can
 synthesize the surface of option prices  and  \cite{eberlein2009sato}   empirically analysed the use of Sato processes in the valuation of equity
structured products.

A class of additive processes with inhomogeneous increments used in finance  are Sato processes.
 \cite{Sa} showed that given  a self
decomposable law $\mu$ an additive process $Y(t)$ always exists  such that $Y(1)\sim\mu$ and  the time $t$ distribution is  the law of
 $t^qY(1)$.
\cite{eberlein2009sato} termed this additive process the Sato process. Sato processes
exhibit a moment term structure in line with the one observed in financial markets, see \cite{boen2019building}.

A possible approach  to include time inhomogeneity in financial models is
 to use  additive subordination; see \cite{mendoza2016multivariate},  \cite{li2016additive} and \cite{kokholm2010sato}. In this case inhomogeneity comes from  the subordinator.

The success of subordination in finance is due to many reasons.
Price processes under no arbitrage are semimartingales and these can be represented as time-changed  Brownian motions.
Furthermore,  time change  models economic time: the more intense the market activity, the faster economic time runs relative
to calendar time.  When the change of time
is a subordinator, the resulting process belongs to the (pure jump) L\'evy class,  a class of analytically tractable processes. In the one-dimensional case, the most famous subordinated Brownian motions such as the variance gamma and the normal inverse Gaussian processes have unit time self-decomposable distributions. Therefore it is possible to define the corresponding Sato processes.
 Unfortunately,   in multivariate subordination this is no longer true.

 In \cite{takano1989mixtures}, the author gave conditions for a multivariate  subordinated Brownian motion to  have  self-decomposable unit time  distribution. The author considered  the subcase of a one-dimensional generalized gamma subordinator.
Even in this case he found that if the  Brownian motion has non-zero drift the subordinated process  distribution at unit time is not self-decomposable.
As a consequence, multivariate subordinated Brownian motions are not good candidates to construct multivariate Sato processes.
However, it is possible to consider multivariate self-decomposable distributions to define multivariate Sato subordinators. Sato subordinators associated to self-decomposable distributions are easy to construct, introduce time inhomogeneity and perform well on financial data (\cite{sun2017marshall}). Therefore,  it is possible to use additive subordination of multivariate Brownian motions to obtain additive processes with inhomogeneous increments to model asset returns.

In this paper we introduce and study  a self-decomposable  class of multivariate exponential tempered distributions and the associated multivariate Sato subordinators. 
This class is defined as a particular case of the tempered distributions in \cite{rosinski2007tempering} and it is a generalization of the multivariate gamma distribution in \cite{perez2014infinitely}. Exponential tempered stable distributions are a multivariate version of the self-decomposable distributions most used in finance, such as the famous CGMY  (\cite{carr2002fine}), the variance gamma (\cite{MS}),  the bilateral gamma (\cite{kuchler2008bilateral}), the gamma  and the inverse Gaussian distributions. We study some properties of multivariate  exponential tempered stable distributions, for example  Proposition \ref{momex}  provides the existence conditions for their moments and Proposition \ref{condSubo} provides conditions for them  to be the  distributions at unit time of  Sato subordinators. We then  characterize  multivariate Sato subordinators by providing their time t L\'evy measure in Theorem \ref{tLmeas} and we focus on a specific  dependence structure widely used in finance to include correlations in multivariate models.

Finally, we build a multivariate additive process using additive multivariate subordination of a Brownian motion.  The construction is designed  to have a multivariate process with the same unit  time  distribution   as  the factor-based $\rho\alpha$-model in \cite{LuciSem1}. This process has  one-dimensional marginal processes that are themselves subordinated Brownian motions. Therefore,  choosing the subordinator properly, the one-dimensional processes are self-decomposable. Making use of this and building on the construction proposed in \cite{guillaume2012sato}, \cite{marena2018multivariate} introduced a multivariate process with the  same unit time distribution of the $\rho\alpha$-models and one-dimensional marginal Sato processes.  Although this process has time inhomogeneous increments, its correlation is constant over time, as in the original L\'evy motion.
For simplicity, correlations are assumed to be constant in many financial models, however this is not a realistic assumption, see e.g. \cite{toth2006increasing}, \cite{teng2016dynamic} and \cite{lundin1998correlation}.

The main contribution of Sato subordination  is to provide time varying correlations, while maintaining the features of the $\rho\alpha$-models dependence structure and  remaining parsimonious in the number of parameters.
Last but not least, as with for the L\'evy and the marginal Sato models, Theorem \ref{SubCh} and its Corollary \ref{chfb} provide the characteristic function in closed form.
Although the empirical investigation of this model is beyond the scope of this work, we conclude with a numerical illustration of the correlation dynamics.

The paper is organized as follows. Section \ref{TS} introduces tempered stable distributions. Exponential tempered stable distributions are studied in Section \ref{ets}, while Section \ref{Msub} defines the corresponding Sato subordinators. Section \ref{Subs} discusses multivariate Sato subordination of a multiparameter Brownain motion. The specification of an asset return model is developed in Section \ref{App} with numerical illustration. Section \ref{conc} concludes.

\section{Tempered stable distributions}\label{TS}

This section introduces  multivariate  stable and tempered stable distributions.

Let $\mu$ be  an infinitely divisible distribution on $\RR^d$  without Gaussian component and $\nu$ its L\'evy measure.
The following proposition provides  the polar decomposition of the L\'evy measure $\nu$ (see e.g. \cite{maejima2009note} and \cite{rosinski1990series}).
\begin{proposition}
Let $\nu$ be a L\'evy measure. Then it exists a  measure  $\lambda$ on $S^{d-1}$ with $0<\lambda(S^{d-1})\leq \infty$ and a family $\{\nu_{\w}: \w\in S^{d-1}\}$ of measures on $(0,\infty)$,  $0<\nu_{\w}(\RR_+)\leq\infty$, such that $\nu_{\w}(E)$ is measurable  in $\w$ for any $E\in \mathcal{B}((0,\infty))$ and it is $\sigma$-finite for any $\w\in S^{d-1}$,
\begin{equation}\label{radLevy}
\int_0^{1}r^2\nu_{\w}(dr)<\infty
\end{equation}
 and
\begin{equation*}
\nu(E)=\int_{S^{d-1}}\lambda{(d\w)}\int_{\RR_+}\boldsymbol{1}_E(s \w)\nu_{\w}(ds),
\end{equation*}
where $\lambda$ and $\nu_{\w}$ are uniquely determined  up to multiplication of measurable functions  $0<c(\w)<\infty$ and $\frac{1}{c(\w)}$, respectively.
\end{proposition}

We say that $\nu$ has  polar decomposition $(\lambda, \nu_{\w})$,  where $\lambda$  and $\nu_{\w}$ are the spherical and  the radial components of $\nu$, respectively. We write $\nu=(\lambda, \nu_{\w})$. Condition \eqref{radLevy} guarantees that $\nu_{\w}$ is a one-dimensional  L\'evy measure for any $\w$. The  radial component of $\mu$ is  the real valued infinitely divisible (i.d.) distribution    $\mu_{\w}$ - without Gaussian component - whose L\'evy measure is $\nu_{\w}$.

%
%
%
%

The measure $\nu$ is said radially absolutely continuous (\cite{Sa}) if it exists a nonnegative measurable function $f(\w, r)$ such that  
\begin{equation*}
\nu(E)=\int_{S^{d-1}}\lambda{(d\w)}\int_{\RR_+}\boldsymbol{1}_E(s \w)f(\w, s)ds.
\end{equation*}

From integration in polar coordinates (see e.g. \cite{folland2013real}), we have that
if $\nu(d\xx)$ is a L\'evy measure with  L\'evy density $f(\xx)$, then it exists a unique Borel measure $\sigma$ on $S^{d-1}$ such that
\begin{equation*}
\begin{split}
\nu(B)=\int_{\RR^d}\boldsymbol{1}_{B}(\xx)f(\xx)d\xx&=\int_{S^{d-1}}\int_{\RR^+}\boldsymbol{1}_{B}(r \w)f(\w r)r^{d-1}dr\sigma(d\w).
\end{split}
\end{equation*}
We call $f_{\w}(r):=f(\w r)r^{d-1}$ the radial component of the density of $\nu$.

%
%
%
%
%
%
%
%
%
%
If $\nu$ is absolutely continuous then $\nu$ is also radially absolutely continuous. The other implication does not hold. It suffices to choose $\lambda$ with finite support and $\nu$ is not absolutely continuous.

Tempered stable distributions are obtained by tempering the radial component of the L\'evy measure of an $\alpha$- stable distribution, $\alpha\in (0,2)$.
It is well known that the L\'evy measure $\nu_0$ of an $\alpha$-stable, $\alpha\in (0,2)$,
measure $\mu_0$ is of the form
\begin{equation}\label{nustable}
\nu_0(E)=\int_{S^{d-1}}\int_{\RR_+}\boldsymbol{1}_E(r \w)\frac{1}{r^{\alpha+1}}dr\lambda(d\w),
\end{equation}
where $\lambda$ is a finite measure on $S^{d-1}$.
Tempered stable distributions are formally defined as follows (\cite{rosinski2007tempering}).
\begin{definition}\label{tstable}
A probability measure $\mu$ on $\RR^d$ is called tempered $\alpha$-stable (abbreviated as T$\alpha$S)
if it is infinitely divisible without Gaussian part and it has L\'evy measure $\nu$ of the following form
\begin{equation}\label{nuAstable}
\nu(dr, d\w)=\frac{q(r, \w)}{r^{\alpha+1}}dr\lambda(d\w),
\end{equation}
where $\alpha\in (0,2)$,  $\lambda$ is a finite measure on $S^{d-1}$ and $q:(0,\infty)\rightarrow (0,\infty)$ is a Borel function such that
$q(\cdot,\w)$ is completely monotone with $\lim_{r\rightarrow\infty}q(r, \w) = 0$. The measure $\mu$ is called a proper T$\alpha$S distribution if, in addition to the above, $q(0+, \w) = 1$ for each $\w\in S^{d-1}$.
\end{definition}
$T\alpha S$ distributions are radially absolutely continuous and belong to the extended Thorin class of infinitely divisible distributions introduced in  \cite{grigelionis2008thorin}: $T^{1-\alpha}(\RR^d)$ for $\alpha\in(0,1)$ and $T^{\chi}(\RR^d)$ fo all $\chi>0$ and $1\leq \alpha<2$.
The radial component $\mu_{\w}$ of a $T\alpha S$ distributions has L\'evy measure  $\nu_{\w}(dr)=\frac{q(r, \w)}{r^{\alpha+1}}dr$.

 Stable and tempered stable distributions are clearly self-decomposable, see Equations \eqref{nustable}, \eqref{nuAstable} and \eqref{sd1}. The notion of self-decomposability is recalled in Appendix \ref{Aself}.

%
%

\section{Exponential  tempered stable distributions}\label{ets}

This section introduces tempered stable distributions with an exponential tempering function $q(r, \w)$. Formally we have the following.

\begin{definition}
An infinitely divisible  probability measure $\mu$ on $\RR^d$ is called  exponential T$\alpha$S ($\ets$) if it is without Gaussian part and it has L\'evy measure $\nu$ of the form
\begin{equation}\label{nuTstable}
\nu(E)=\int_{S^{d-1}}\int_{\RR_+}\boldsymbol{1}_E(r \w)\frac{e^{-\beta( \w)r}}{r^{\alpha+1}}dr\lambda(d\w),
\end{equation}
where $\alpha\in[0, 2),$  $\lambda$ is a finite measure on $S^{d-1}$ and $\beta: S^{d-1}\rightarrow (0, \infty)$ is a Borel-measurable function.
\end{definition}
If $\mu$ has L\'evy measure \eqref{nuTstable} we write $\mu\sim\ets(\alpha, \beta, \lambda)$. We include the case
 $\alpha=0$ because it  gives the multivariate gamma distribution introduced and studied in \cite{perez2014infinitely}, that is of interest in financial applications. However,  we  focus on $\alpha\in (0,2)$ and refer to  \cite{perez2014infinitely} for the case $\alpha=0$.
The following proposition gives conditions for a $\ets$ distribution to exists and characterizes its L\'evy measure.
 \begin{proposition}\label{variations}
 Equation  \eqref{nuTstable} defines a L\'evy measure and therefore it exists an $\ets$ distribution.
\begin{enumerate}

\item if $\alpha\in [0,1)$, it holds $\int_{||x||\leq1}||\xx|| \nu(d\xx)<\infty$ and $\int_{\RR^d} \nu(d\xx)=\infty$;
\item  if $\alpha\in [1,2)$,  it holds $\int_{||x||\leq1}||\xx|| \nu(d\xx)=\infty$.
\end{enumerate}
\end{proposition}
\begin{proof}
Let  $\nu=(\lambda, \nu_{\w})$ be the L\'evy measure in \eqref{nuTstable}
and let $B=\{\xx\in \RR^d:\,\, ||\xx||\leq 1\}$.
Since
\begin{equation*}
\begin{split}
\int_{B}||\xx||^2\nu(d\xx) &=\int_{S^{d-1}}\int_{0}^{1}\frac{r^2e^{-\beta( \w)r}}{r^{\alpha+1}}dr\lambda(d\w)\leq \int_{S^{d-1}}\int_{0}^1{r^{1-\alpha}}dr\lambda(d\w) \\
&=\int_{S^{d-1}}\frac{1}{2-\alpha}\lambda(d\w)=\frac{1}{2-\alpha}\lambda(S^{d-1})<\infty,\\
\end{split}
\end{equation*}
and
\begin{equation*}
\begin{split}
\int_{S^{d-1}}\int_{1}^{\infty}\frac{e^{-\beta( \w)r}}{r^{\alpha+1}}dr\lambda(d\w)&\leq \int_{S^{d-1}}\int_{1}^{\infty}\frac{1}{r^{\alpha+1}}dr\lambda(d\w)\\
&= \int_{S^{d-1}}\frac{1}{\alpha}\lambda(d\w)=\frac{1}{\alpha}\lambda(S^{d-1})<\infty,\\
\end{split}
\end{equation*}
we have $\int_{\RR^d}(||\xx||^2\wedge1)\nu(d\xx) <\infty$ and $\nu$ is a L\'evy measure.

\begin{enumerate}

\item See \cite{perez2014infinitely} for $\alpha=0$ and let  $\alpha\in (0,1)$ . It holds:
\begin{equation*}
\begin{split}
\int_{B}||\xx|| \nu(d\xx)&=\int_{S^{d-1}}\int_{0}^1r\frac{e^{-\beta( \w)r}}{r^{\alpha+1}}dr\lambda(d\w)\\
&\leq \int_{S^{d-1}}\int_{0}^1\frac{1}{r^{\alpha}}dr\lambda(d\w)= \int_{S^{d-1}}\frac{1}{1-\alpha}\lambda(d\w)=\frac{1}{1-\alpha}\lambda(S^{d-1}).
\end{split}
\end{equation*}
The infinite activity of $\nu$ follows from the infinite activity of its radial component.
it holds:
\begin{equation*}
\begin{split}
\int_{B}||\xx|| \nu(d\xx)&=\int_{S^{d-1}}\int_{0}^1r\frac{e^{-\beta( \w)r}}{r^{\alpha+1}}dr\lambda(d\w)\\
&=\int_{S^{d-1}}\int_{0}^1\frac{e^{-\beta( \w)r}}{r^{\alpha}}dr\lambda(d\w)\geq\int_{S^{d-1}}e^{-\beta(\w)}\int_{0}^1\frac{1}{r^{\alpha}}dr\lambda(d\w),
\end{split}
\end{equation*}
and $\int_{0}^1\frac{1}{r^{\alpha}}dr$ diverges if $\alpha\geq1$. Therefore $\int_{B}||\xx|| \nu(d\xx)=\infty$.

\end{enumerate}

\end{proof}
\begin{remark}
The $\ets$
   distributions are proper $T\alpha S$ distributions (see Definition \ref{tstable}), they are self-decomposable and they are radially absolutely continuous.
\end{remark}
If $\mu\sim\ets(\alpha, \beta, \lambda)$ and  $\beta$ is constant we say that the measure $\mu$ and its L\'evy measure $\nu$ are  homogeneous.

\begin{corollary}

The characteristic function of an $\ets$ distribution has the form:
\begin{equation}\label{sdcart2}
\hmu(\zz)=\exp\{i\boldsymbol{\gamma}\cdot \zz+\int_{\RR^d}(e^{i\langle\zz,\xx\rangle}-1-i\langle\zz,\xx\rangle\boldsymbol{1}_{||\xx||\leq 1}(\xx))\frac{e^{-\beta{(\frac{\xx}{||\xx||})} ||\xx||}}{||\xx||^{\alpha+1}}\tilde{\lambda}(d\xx)\},
\end{equation}

where $\tilde{\lambda}$ has the form of \eqref{tildelambda}.
\end{corollary}
\begin{proof}
From \eqref{nuTstable} and \eqref{sd1} $\mu$ is self-decomposable with
 $k_{\w}(r)=\frac{e^{-\beta( \w)r}}{r^{\alpha}}$, thus it is self-decomposable with $h$ function given by
 $h(\xx)=k_{\frac{\xx}{||\xx||}}(||\xx||)=\frac{e^{-\beta(\frac{\xx}{||\xx||})||\xx||}}{||\xx||^{\alpha}}$, thus \eqref{sdcart2} follows.
\end{proof}

If $\nu$ is a one-dimensional $\ets$ L\'evy measure from equation \eqref{sdcart2} we have
\begin{equation*}
\nu(dx)=(\boldsymbol{1}_{(-\infty,0)}(x)\frac{e^{-\beta^+ |x|}}{|x|^{\alpha+1}}\lambda^++\boldsymbol{1}_{(0,\infty)}(x)\frac{e^{-\beta^- x}}{x^{\alpha+1}}\lambda^-)dx,
\end{equation*}
where $\beta^+=\beta(1)$, $\beta^-=\beta(-1)$, $\lambda^+=\lambda(\{1\})$ and $\lambda^-=\lambda(\{-1\})$. Thus, $\ets$ distributions are  multivariate versions of the  tempered stable distributions studied in \cite{kuchler2013tempered}, with the restriction that $\alpha$ is constant. By properly specifying the parameters, we find multivariate versions of well known tempered stable distributions as the CGMY, the variance gamma, the inverse Gaussian and the gamma ($\alpha=0$) distributions. We also observe that  the one-dimensional  radial component   $\mu_{\w}$ of an $\ets$ distribution  belongs to the class of one sided tempered stable distributions studied in \cite{kuchler2013tempered}.

The exponential tilting  $q(r, \w)=e^{-\beta(\w)r}$ is an inhomogeneous tilting of a stable L\'evy measure in all directions. In fact fo any $\w\in S^{d-1}$, $\nu_{\w}(dr)=e^{-\beta(\w) r}\nu_0(dr)$, where $\nu_0(dr)$ is a one-dimensional $\alpha$-stable distribution.
Tempered stable distributions have been characterized in \cite{rosinski2007tempering} in terms of their spectral measure $R$.
We recall the definition of spectral measure and then we find it for $\ets$ distributions.

Let $\mu\in \ets(\alpha, \beta, \lambda)$. The tempering function of $\mu$   can be represented as
\begin{equation*}
e^{-\beta(\w)r}=\int_{\RR_+}e^{-rs}Q(ds|\w),
\end{equation*}
where  $Q(ds|\w)=\delta_{\beta(\w)}(s)ds$- $\ets$ are proper tempered stable distributions.
Let us introduce the measure $Q$:
\begin{equation*}
\begin{split}
Q(A)&=\int_{S^{d-1}}\int_{\RR_+}\boldsymbol{1}_A(r \w)Q(dr|\w)\lambda(d\w).
\end{split}
\end{equation*}
Clearly in this case the measure  $Q$ becomes
\begin{equation}\label{qa}
\begin{split}
Q(A)&=\int_{S^{d-1}}\int_{\RR_+}\boldsymbol{1}_A(r \w)Q(dr|\w)\lambda(d\w)=\int_{S^{d-1}}\boldsymbol{1}_A(\beta(\w) \w)\lambda(d\w)=\int_{S^{d-1}}\boldsymbol{1}_{A_{\beta}}(\w)\lambda(d\w),
\end{split}
\end{equation}
where $A_{\beta}=\{\w\in S^{d-1}: \beta(\w) \w\in A\}$.
The spectral measure  $R$ is defined from $Q$ as
\begin{equation*}
\begin{split}
R(A)&:=\int_{\RR^d}\boldsymbol{1}_A(\frac{\xx}{||\xx||^2})||\xx||^{\alpha}Q(d\xx),\\
\end{split}
\end{equation*}
in this case we have
\begin{equation}\label{R}
\begin{split}
R(A)
&=\int_{S^{d-1}}\boldsymbol{1}_A(\frac{\w}{\beta(\w)})\beta(\w)^{\alpha}\lambda(d\w).
\end{split}
\end{equation}
Therefore
\begin{equation*}
\begin{split}
R(\RR^d)&=\int_{S^{d-1}}\beta(\w)^{\alpha}\lambda(d\w).
\end{split}
\end{equation*}
Furthermore, from \eqref{qa} we have $Q(\RR^d)=\lambda(S^{d-1})$. We also have
\begin{equation}
\begin{split}
\int_{\RR^{d}}||\xx||^{\alpha}R(d\xx)&=\int_{S^{d-1}}||\frac{\w}{\beta(\w)}||^{\alpha}\beta(\w)^{\alpha}\lambda(d\w)=
\int_{S^{d-1}}\lambda(d\w)=\lambda(S^{d-1}).
\end{split}
\end{equation}
Finally,
\begin{equation*}
\begin{split}
\lambda(S^{d-1})=Q(\RR^d)=\int_{\RR^{d}}||\xx||^{\alpha}R(d\xx).
\end{split}
\end{equation*}

The existence of the moments of tempered stable distributions depends on their spectral measure, as proved in Proposition 2.7, \cite{rosinski2007tempering}. Therefore, by \eqref{R}, the existence of the moments of an $\ets$ distributions depends on the function $\beta$ and on the spherical components $\lambda$ of its L\'evy measure, as stated in the following proposition.

\begin{proposition}\label{momex}
Let  $\mu\sim\ets(\alpha, \beta, \lambda)$, $\alpha\in (0, 2)$ then we have
\begin{enumerate}
\item $\int_{\RR^d}||\xx||^kd\mu(\xx)<\infty$ for $k\in (0,\alpha)$;
\item  $\int_{\RR^d}||\xx||^{\alpha}d\mu(\xx)<\infty$ if and only if $\int_{\{\w: \beta(\w)>1\}}{\beta(\w)}^{\alpha}\log(\beta(\w))\lambda(d\w)<\infty$;
\item $\int_{\RR^d}||\xx||^kd\mu(\xx)<\infty$ if and only if $\int_{S^{d-1}}{\beta(\w)}^{-k-\alpha}\lambda(d\w)<\infty$, for $k>\alpha$.
\end{enumerate}

\end{proposition}
\begin{proof}
\begin{enumerate}
\item Follows from Proposition 2.7, \cite{rosinski2007tempering}.
\item Clearly
\begin{equation*}
\begin{split}
\int_{||\xx||>1}||\xx||^{\alpha}\log(||\xx||)R(dx)&=\int_{S^{d-1}}1_{\{\beta(\w)>1\}}{\beta(\w)}^{\alpha}\log(\beta(\w))\lambda(d\w).\\
\end{split}
\end{equation*}
and then Proposition 2.7 \cite{rosinski2007tempering} applies.
\item The condition
$
\int_{\RR^d}||\xx||^kd\mu(\xx)
$ is equivalent to  $
\int_{\{||\xx||>1\}}||\xx||^k\nu(d\xx)<\infty
$ (\cite{Sa} p.159)
that is equivalent to
\begin{equation*}
\begin{split}
\int_{\{||\xx||>1\}}||\xx||^k\nu(d\xx)&=\int_{S^{d-1}}\int_1^{\infty}e^{-\beta(\w)r}r^{k-\alpha-1}dr\lambda(d\w)
<\infty.
\end{split}
\end{equation*}
The latter is equivalent to $\int_{S^{d-1}}{\beta(\w)}^{-k-\alpha}\lambda(d\w)<\infty$, as proved in \cite{perez2014infinitely}, Proposition 3.12.
\end{enumerate}

\end{proof}

Notice that, in particular, if $\lambda$ has finite support all its moments exist. 
%
%
%
%
%
%
%
%
%
%
%
%
%
%
%
The next proposition provides the characteristic function of $\mathcal{E}T\alpha S$ distributions.
\begin{proposition}\label{chalpha}

The characteristic function $\hat{\mu}$   of $\mu\sim$$\ets$$(\alpha, \lambda, \beta)$  is:

\noindent for $\alpha\in (0, 1)$
\begin{equation}\label{ppbv1}
\hat{\mu}( \zz)=exp\{\Gamma(-\alpha)\int_{S^{d-1}}[(\beta(\w)-i\langle\w, \zz \rangle)^{\alpha}-\beta(\w)^{\alpha}]\lambda(d \w)\}, \,\,\, \forall  \zz \in \RR^d,
\end{equation}
for $\alpha\in (1, 2)$
\begin{equation}\label{ppuv2}
\hat{\mu}( \zz)=exp\{\Gamma(-\alpha)\int_{S^{d-1}}[(\beta(\w)-i\langle\w, \zz \rangle)^{\alpha}-\beta(\w)^{\alpha}+i\langle\w, \zz \rangle\alpha\beta(\w)^{\alpha-1}]\lambda(d \w)\}, \,\,\, \forall  \zz \in \RR^d,
\end{equation}
and for $\alpha=1$
\begin{equation}\label{alph1}
\hat{\mu}( \zz)=exp\{\int_{S^{d-1}}[(\beta(\w)-i\langle \w, \zz\rangle)\log(\beta(\w)-i\langle \w, \zz\rangle)+i\langle\w, \zz \rangle]\lambda(d \w)\}, \,\,\, \forall  \zz \in \RR^d.
\end{equation}

\end{proposition}
\begin{proof}
The L\'evy Khintchine formula in polar coordinates with \eqref{nuTstable} gives
\begin{equation}\label{pp}
\hat{\mu}( \zz)=exp\{\int_{S^{d-1}}\phi_{\w}(\langle \w, \zz \rangle)\lambda(d \w)\}, \,\,\, \forall  \zz \in \RR^d,
\end{equation}
where if $\alpha\in(0,1)$ we have
\begin{equation*}
\phi_{\w}(\langle \w, \zz \rangle):=\int_{\RR_+}(e^{ir\langle \w, \zz\rangle}-1)\frac{e^{-\beta( \w)r}}{r^{\alpha+1}}dr,
\end{equation*}
because  by Proposition \ref{variations} it holds $\int_{|x|\leq1}||\xx|| \nu(dx)<\infty$.  If $\alpha\in(1,2)$ we have
\begin{equation*}
\phi_{\w}(\langle \w, \zz \rangle):=\int_{\RR_+}(e^{ir\langle \w, \zz\rangle}-1-ir\boldsymbol{1}_{|r|<1}(r)\langle \w, \zz \rangle)\frac{e^{-\beta( \w)r}}{r^{\alpha+1}}dr.
\end{equation*}

Therefore for each $\w$, $\phi_{\w}(\langle \w, \zz \rangle)$  is the characteristic exponent of the one sided $\ets$ distribution $\mu_{\w}$ that are provided in \cite{kuchler2013tempered}.

If $\alpha\in(0,1)$, for each $\w$ the characteristic function of the radial component $\mu_{\w}$ of $\mu$  is:
\begin{equation*}
\phi_{\w}(\langle \w, \zz \rangle)=\Gamma(-\alpha)[(\beta(\w)-i\langle\w, \zz \rangle)^{\alpha}-\beta(\w)^{\alpha}]
\end{equation*}
and \eqref{ppbv1} follows.
If $\alpha\in(1,2)$,  the characteristic function of the radial component $\mu_{\w}$ of $\mu$  is:
\begin{equation*}
\phi_{\w}(\langle \w, \zz \rangle)=\Gamma(-\alpha)[(\beta(\w)-i\langle\w, \zz \rangle)^{\alpha}-\beta(\w)^{\alpha}+i\langle\w, \zz \rangle\alpha\beta(\w)^{\alpha-1}]
\end{equation*}
and \eqref{ppuv2} follows.

If $\alpha=1$,   from Theorem 2.9 in \cite{rosinski2007tempering} we have:

\begin{equation*}
\hat{\mu}( \zz)=exp\{\int_{\RR^d}\Phi(\langle \xx, \zz\rangle) R(d \xx)\}, \,\,\, \forall  \zz \in \RR^d,
\end{equation*}
where
\begin{equation*}
\Phi(\langle \xx, \zz\rangle)=(1-i\langle \xx, \zz\rangle)\log(1-i\langle \xx, \zz\rangle)+i\langle\xx, \zz \rangle, \,\,\, \forall  \zz \in \RR^d,
\end{equation*}
From \eqref{R} it follows

\begin{equation*}
\hat{\mu}( \zz)=exp\{\int_{S^{d-1}}\Phi(\langle \frac{\w}{\beta(\w)}, \zz\rangle) \lambda(d \w)\}, \,\,\, \forall  \zz \in \RR^d,
\end{equation*}
that gives \eqref{alph1}.
\end{proof}
%
%
%
%
%
%
%
%

The following  propositions allow us  to construct multivariate $\ets$ distributions useful in applications, as we do in the next section.
\begin{proposition}\label{sumind}
A measure   $\mu\sim\ets(\alpha, \beta, \lambda)$  is absolutely continuous  if and only if the support of $\lambda$ contains  $d$ linearly independent vectors $\w_j, j=1,\ldots, d$. Let $\XX\sim \mu$, $\XX$ has independent components
in and only if $\lambda$ has support on $\boldsymbol{e}_i, i=1,\ldots, d$ in $\RR^d$.
\end{proposition}
\begin{proof}
If the support of $\lambda$ contains  $d$ linearly independent vectors, then the support of $\nu$ is full dimension. Therefore $\mu$ is a genuinely $d$-dimensional self-decomposable distribution, then $\mu$ is absolutely continuous (\cite{sato1982absolute}).
Let $\XX\sim \mu$. Then $\XX$ has independent components if and only if $\nu$ has support on the coordinate axes   (see e.g. \cite{Sa}, E 12.10).

\end{proof}

\begin{proposition}\label{closure}
Let $\alpha\in(0,1)$.
Let $\XX_1\sim$ $\ets$ $(\alpha, \beta, \lambda_1)$,  $\XX_2\sim$ $\ets$ $(\alpha, \beta, \lambda_2)$ and let them be independent. Then
\begin{enumerate}
\item $\XX_1+\XX_2\sim$ $\ets$ $(\alpha, \beta, \lambda_1+ \lambda_2)$;
\item For a constant  $c>0$, $c\XX_1\sim$ $\ets$ $(\alpha, \frac{\beta}{c},  c^{\alpha}\lambda_1)$.

\end{enumerate}
\end{proposition}

\begin{proof}

\begin{enumerate}
\item  By independence
\begin{equation*}
\begin{split}
\hat{\mu}_{\XX_1+\XX_2}(\zz)&= \hat{\mu}_{\XX_1}(\zz) \hat{\mu}_{\XX_2}(\zz)
=exp\{\Gamma(-\alpha)\int_{S^{d-1}}\phi_{\w}(\langle \w, \zz \rangle)(\lambda_1+\lambda_2)(d \w)\}
\end{split}
\end{equation*}
\item 
Since (see \cite{kuchler2013tempered})
\begin{equation*}
\begin{split}
\phi_{\w}(\langle \w, c\zz \rangle)&=\Gamma(-\alpha)[(\beta(\w)-i\langle\w, c\zz \rangle)^{\alpha}-\beta(\w)^{\alpha}]\\
&=\Gamma(-\alpha)c^{\alpha}[(\frac{\beta(\w)}{c}-i\langle\w, c\zz \rangle)^{\alpha}-(\frac{\beta(\w)}{c})^{\alpha}],
\end{split}
\end{equation*}
from \eqref{pp} we have the assert.

\end{enumerate}
\end{proof}
\begin{proposition}
Let $\alpha\in (0,1)$. The mean vector $\boldsymbol{m}$ and the covariance matrix $\Sigma$ of  $\mu\sim\ets(\alpha, \beta,\lambda)$  are
\begin{equation*}
\begin{split}
\boldsymbol{m}=\int_{S^{d-1}}\Gamma(1-\alpha)\beta(\w)^{\alpha-1}\w\lambda(d\w)\\
\end{split}
\end{equation*}
and
\begin{equation*}
\begin{split}
\Sigma=\int_{S^{d-1}}\Gamma(2-\alpha)\beta(\w)^{\alpha-2}\w\w^T\lambda(d\w)\\
\end{split}
\end{equation*}
\begin{proof}
The cumulant generating function exists on $\{\zz\in \RR^n: \langle\w, \zz \rangle\leq \beta(\w)\}$ and it is
\begin{equation*}
\begin{split}
k(\zz)&=\int_{S^{d-1}}\Gamma(-\alpha)[(\beta(\w)-\langle\w, \zz \rangle)^{\alpha}-\beta(\w)^{\alpha}]\lambda(d\w)
\end{split}
\end{equation*}
We have $m_j=\frac{\partial}{\partial z_j}k(\zz)|_{\zz=0}$ and $\Sigma_{ij}=\frac{\partial}{\partial z_i\partial z_j}k(\zz)|_{\zz=0}$, the thesis follows by inverting integration ad differentiation.
\end{proof}

\end{proposition}

\subsection{Specifications}

As mentioned in the previous section, multivariate $\ets$ distributions are multivariate extension of tempered stable distributions in \cite{kuchler2013tempered}. Therefore by properly choosing the parameters we have the following multivariate distributions.
%

\begin{enumerate}

\item {\bf Multivariate CGMY distribution.}
A Multivariate $CGMY(C, \beta, \alpha)$  distribution is a distribution $\mu$   with L\'evy measure in \eqref{nuTstable}, where   $\lambda(d\w)=C \sigma(d\w)$ and $\sigma$  is the unique measure induced on $S^{d-1}$ from the Lebesgue measure on $\RR^d$.
We have
\begin{equation*}
\nu(d\xx)=C\frac{h(\xx)}{||\xx||^d}d\xx,
\end{equation*}
where
$h(\xx)=k_{\frac{\xx}{||\xx||}}(||\xx||)=\frac{e^{-\beta(\frac{\xx}{||\xx||})||\xx||}}{||\xx||^{\alpha}}$,
therefore
\begin{equation*}\label{nuXx1}
\nu(E)=C\int_{\RR^{d}}\boldsymbol{1}_E(\xx)\frac{e^{-\beta{(\frac{\xx}{||\xx||})} ||\xx||}}{||\xx||^{\alpha+d}}d\xx,
\end{equation*}
If  $\mu$ is  one-dimensional we have
\begin{equation*}\label{nuCGMY}
\nu(dx)=C(\boldsymbol{1}_{(-\infty,0)}(x)\frac{e^{-G |x|}}{|x|^{\alpha+1}}+\boldsymbol{1}_{(0,\infty)}(x)\frac{e^{-Mx}}{x^{\alpha+1}})dx,
\end{equation*}
thus $\mu$ is a CGMY distribution.
We have $G=\beta(-1)$, $M=\beta(1)$ and $\alpha$ is the parameter $Y$, where $CGMY$ are the original parameters in  \cite{carr2002fine}.
As one can see, as in the univariate case,  parameter $C$ may be viewed as a measure of the overall level of activity and function $\beta$ controls the rate of exponential decay in each direction.
The homogeneous case $\beta=$cost is  the multivariate extension of the symmetric CGMY, i.e. $G=M$.

\item
{\bf Multivariate bilateral gamma and variance gamma distributions.}
The case $\alpha=0$  includes multivariate versions of the bilateral gamma distribution introduced in \cite{kuchler2008bilateral} and  of  the famous  variance gamma distribution \cite{MS}. This case has been introduced and studied in \cite{perez2014infinitely} under the name of multivariate Gamma distribution. We therefore refer to them for the  case $\alpha=0$ .

\item
{\bf Multivariate  gamma distribution.}
The case $\alpha=0$ with the further condition $\lambda(S^{d}_-)=0$ - also studied in \cite{perez2014infinitely} - is a multivariate gamma distribution in a narrow sense, see also \cite{semeraro2020note}.
\item{\bf Multivariate inverse Gaussian distribution.}
A  $d$-dimensional  inverse Gaussian distribution $\mu$  with parameters $\lambda$ and $\beta$, denoted  ${IG}(\lambda, \beta)$, is
 a $\ets$ distribution with the following L\'evy measure:
\begin{equation*}\label{nuXx}
\nu_{IG}(E)=\int_{S^{d-1}}\lambda(d\w)\int_{\RR_+}\boldsymbol{1}_E(r\w)\frac{e^{-\beta({\w})r}}{r^{3/2}}dr,
\end{equation*}
where $\forall \w\in S^{d-1}$, $\nu_{\w}(dr)=\frac{e^{-\beta({\w})r}}{r^{3/2}}dr$ is the L\'evy measure of a one-dimensional IG process with parameters $(1, \beta({\w}))$.
If $\lambda$ has support in $S_+^{d-1}$ then we have a multivariate $IG$ distribution in a narrow sense.

The inverse Gaussian distribution is the specification we focus on in the application proposed. Therefore we provide
its characteristic function that is easily derived from of Proposition \ref{chalpha}.
\begin{equation*}
\hat{\mu}( \zz)=exp\{-2\sqrt{\pi}\int_{S^{d-1}}[\sqrt{(b^2(\w)-i\langle\w, \zz \rangle)}-b(\w)]\lambda(d \w)\}, \,\,\, \forall  \zz \in \RR^d,
\end{equation*}
where $b(\w)=\sqrt{\beta(\w)}$.
As in the one-dimensional case if $\lambda(d\w)=C \sigma(d\w)$ and $\sigma$  is the unique measure induced on $S^{d-1}$ from the Lebesgue measure on $\RR^d$, the  inverse Gaussian distribution is a special case of  $CGMY$  distribution  with $\alpha=\frac{1}{2}$.

\end{enumerate}

\section{Multivariate $\ets$ Sato subordinators}\label{Msub}

This section  introduces time inhomogeneous  additive subordinators with unit time  $\ets$ distribution.  An additive subordinator is an increasing process with non stationary independent increments.
The characteristic function of an additive subordinator $\SS(t)$ is  (\cite{mendoza2016multivariate})
\begin{equation}\label{CHSbv}
\hat{\mu}_t(\zz)=\exp\{i \langle \boldsymbol{\gamma}(t), \zz\rangle+\int_0^t\int_{\RR_+^d}(e^{i\langle \xx,\zz\rangle}-1)g(d\xx, u) du\},
\end{equation}
where $\boldsymbol{\gamma}(t)\in \RR^d_+$ is the time dependent drift and  $g(d\xx, u)$ is  a  time-dependent measure  so that $\int_{0}^1||\xx||g(d\xx, u)<\infty$  for almost all $u$. We assume $\boldsymbol{\gamma}(t)=0$  and we call $g$ the differential L\'evy measure   of $\SS(t)$ according to \cite{li2016additive}.

Since a measure $\mu\sim\ets(\alpha, \beta, \lambda)$ is self-decomposable, we can define a  Sato process associated with $\mu$, that we call $\ets$ Sato process (see Appendix \ref{Aself} for the definition of a Sato process associated to a self-decomposable distribution).
 The time $t$ characteristic function of a Sato process  $\SS(t)\sim\mu_t$  is given by
\begin{equation}\label{Chscaled}
\hat{\mu}_t( \zz)=\hat{\mu}(  t^q\zz), \,\,\, \forall  \zz \in \RR^d,
\end{equation}
where $\mu$ is a self-decomposable distribution. Sato processes are additive precesses and therefore we can define an additive Sato subordinator as follows.
\begin{definition}

A $d$-dimensional Sato subordinator is a  Sato process with one-dimensional positive and increasing trajectories.
\end{definition}

We study multivariate Sato subordinators $\SS(t)$ with unit time  $\ets$ distributions $\mu$, i.e. $\SS(1)\sim \mu$. The next proposition gives condition for an $\ets$ distribution to be the self-decomposable distribution associated to a Sato subordinator.

\begin{proposition}\label{condSubo}
The L\'evy measure in \eqref{nuTstable}  is the L\'evy measure of an $\ets$ Sato subordinator $\SS(t)$ if and only if
$\alpha\in (0,1)$  and $\lambda$ has support on $S^{d-1}_+$.
\end{proposition}

\begin{proof}
For any $\w\in S^{d-1}$ let  $S_{\w}(t)$ be a  one-dimensional Sato process  with L\'evy measure $\nu_{\w}.$
The process  $S_{\w}(t)$ only  has  positive jumps because $\nu_{\w}$ is a positive L\'evy measure. Therefore
 $S_{\w}(t)$ is an increasing process if and only if $\int_{(0,1]}x\nu_{\w}(dx)<\infty$, see \cite{Sa}, thus if and only if $\alpha\in (0,1)$. Finally, the L\'evy measure $\nu$ is positive if and only if $\lambda$ has support on $S^{d-1}_+$.
%
%

\end{proof}


\begin{theorem}\label{tLmeas}
Let $\mu\sim$ $\ets(\alpha, \beta, \lambda)$,   the associated Sato subordinator  $\SS(t)$   has
 time t L\'evy measure given by
\begin{equation}\label{SatoLM}
\nu(E, t)=\int_{S_+^{d-1}}\int_{\RR_+}\boldsymbol{1}_E(r \w)\frac{e^{-\beta( \w)rt^{-q}}}{r^{\alpha+1}}t^{\alpha q}dr\lambda(d\w)
\end{equation}
and it has characteristic function \eqref{CHSbv} with
differential L\'evy measure
\begin{equation}\label{diffL}
\begin{split}
g(E, u)&=\int_{S_+^{d-1}}\int_{\RR_+}\boldsymbol{1}_E(r\w)e^{-\beta( \w)ru^{-q}}qu^{\alpha q-1}\frac{\beta(\w)ru^{-q}+\alpha}{r^{\alpha+1}}dr\lambda(d\w).
\end{split}
\end{equation}

\end{theorem}
\begin{proof}

The time $t$ characteristic function of an $\ets$-Sato subordinator with zero drift is given by \eqref{Chscaled}, therefore

\begin{equation*}
\begin{split}
\hat{\mu}_t( \zz)&=\hat{\mu}(t^q\zz)=exp\{\int_{S^{d-1}}\int_{\RR_+}(e^{ir\langle \w, t^q\zz\rangle}-1)\frac{e^{-\beta( \w)r}}{r^{\alpha+1}}dr\lambda(d \w)\}\\
&=exp\{\int_{S^{d-1}}\int_{\RR_+}(e^{irt^q\langle \w, \zz\rangle}-1)\frac{e^{-\beta( \w)r}}{r^{\alpha+1}}dr\lambda(d \w)\}\\
&=exp\{\int_{S^{d-1}}\int_{\RR_+}(e^{iu\langle \w, \zz\rangle}-1)\frac{e^{-\beta( \w)ut^{-q}}}{u^{\alpha+1}t^{-\alpha q}}du\lambda(d \w)\}, \,\,\, \forall  \zz \in \RR^d.
\end{split}
\end{equation*}
and the time $t$ L\'evy measure is \eqref{SatoLM}.
Let now
\begin{equation*}\label{diffL2}
\begin{split}
g_{\w}(r, u)&=\frac{\partial \nu_{\w}(r, u)}{\partial u}=e^{-\beta( \w)ru^{-q}}qu^{\alpha q-1}\frac{\beta(\w)ru^{-q}+\alpha}{r^{\alpha+1}}.
\end{split}
\end{equation*}
We have

\begin{equation*}
\begin{split}
\hat{\mu}_t( \zz)&=exp\{\int_{S^{d-1}}\int_{\RR_+}(e^{ir\langle \w, t^q\zz\rangle}-1)\int_0^te^{-\beta( \w)ru^{-q}}qu^{\alpha q-1}\frac{\beta(\w)ru^{-q}+\alpha}{r^{\alpha+1}}dudr\lambda(d \w)\}
\end{split}
\end{equation*}
 Fubini-Tonelli applies and we have

\begin{equation*}
\begin{split}
\hat{\mu}_t( \zz)&=exp\{\int_0^t\int_{S^{d-1}}\int_{\RR_+}(e^{ir\langle \w, t^q\zz\rangle}-1)e^{-\beta( \w)ru^{-q}}qu^{\alpha q-1}\frac{\beta(\w)ru^{-q}+\alpha}{r^{\alpha+1}}dr\lambda(d \w)du\},
\end{split}
\end{equation*}
that with  \eqref{CHSbv} gives that
$
g_{\w}(r,u)
$ is the density of the radial component of the differential L\'evy measure, therefore \eqref{diffL} follows.

%
%
%
\end{proof}

\subsection{Examples}
We consider here the Sato subordinator associated to the multivariate gamma distribution  and to the multivariate  IG distribution.

\begin{enumerate}
\item {\bf Multivariate Sato-gamma.}
The multivariate Sato gamma subordinator is the Sato subordinator associated to the multivariate gamma distribution in \cite{perez2014infinitely} .
A Sato-Gamma subordinator is a Sato process $\SS(t)$ such that $\SS(1)\sim \ets(0, \beta, \lambda)$ and  $\lambda$ has support on $S^{d-1}_+$.

Notice that if $\lambda$ has support on $S^{d-1}$, we have a multivariate Sato-bilateral gamma or a multivariate Sato-variance gamma process. A process in this class has been used  by \cite{boen2019towards} to model asset prices.

\item{ \bf Multivariate Sato-inverse Gaussian.}
A Sato-inverse gamma (S-IG) subordinator is a Sato process $\SS(t)$ such that $\SS(1)\sim \ets(\frac{1}{2}, \beta, \lambda)$ and  $\lambda$ has support on $S^{d-1}_+$. We write $\SS(t)\sim$ S-IG$(\beta, \lambda)$.
The time $t$ L\'evy measure of a S-IG subordinator is
\begin{equation*}
\nu_{IG}(E, t)=\int_{S_+^{d-1}}\int_{\RR_+}\boldsymbol{1}_E(r \w)\frac{t^{\frac{q}{2}}e^{-\beta( \w)rt^{-q}}}{r^{q+1}}dr\lambda(d\w).
\end{equation*}
%
The S-IG subordinator is used in Section \ref{App} for our application to finance. Therefore we report its time $t$ characteristic function. Let $\SS(t)\sim$ S-IG$(\beta, \lambda)$, then its characteristic function is
\begin{equation*}
\hat{\mu}_t( \zz)=\hat{\mu}( t^q\zz)=exp\{-2\sqrt{\pi}\int_{S_+^{d-1}}[\sqrt{(b^2(\w)-i\langle\w, t^q\zz \rangle)}-b(\w)]\lambda(d \w)\}, \,\,\, \forall  \zz \in \RR^d,
\end{equation*}
where $b(\w)=\sqrt{\beta(\w)}$.

If $\lambda$ has support on $S^{d-1}$ we have  a multivariate Sato inverse Gaussian process.
\end{enumerate}

\section{Sato-$\ets$ subordinated Brownian motion}\label{Subs}
In this section we build a multivariate inhomogeneous additive process by subordinating a multiparameter Brownian motion with a multivariate $\ets$ Sato subordinator. For the formal definition of multiparameter  (L\'evy)  process we refer to \cite{Ba}. The Sato subordinator is assumed to have zero drift, i.e. $\gamma(0)=0$ in \eqref{CHSbv}. This assumption allows us to avoid the introduction of regularized Sato $\ets$ subordinators (see \cite{li2016additive}).

Let ${\boldsymbol{B}}_{i}(t)$ be independent Brownian motions on
$\mathbb{R}^{n_{i}}$ with drift $\boldsymbol{\mu}$ and covariance matrix
$\boldsymbol{\Sigma_{i}}$, and let $\boldsymbol{B}=\{{\boldsymbol{B}}(\boldsymbol{s}), \boldsymbol{s}%
\in\mathbb{R}^{d}_{+}\}$, where $\boldsymbol{B}(\boldsymbol{s}) := (\boldsymbol{B}_{1}(s_{1}), \ldots ,\boldsymbol{B}_{d}(s_{d}) )^{T}$,  be the associated multiparameter L\'{e}vy process.
Let $\boldsymbol{A}_{i}\in\mathcal{M}_{n\times{n_{i}}}(\mathbb{R})$. We can define the
process $\boldsymbol{B}_{\boldsymbol{A}}=\{\boldsymbol{B}_{\boldsymbol{A}}(\boldsymbol{s}),\boldsymbol{s}\in\mathbb{R}^{d}_{+}\}$ as%
\begin{equation}
\label{Brho1}{\boldsymbol{B}}_{\boldsymbol{A}}(\boldsymbol{s})=\boldsymbol{A}_{1}{{\boldsymbol{B}_{1}}%
}({s}_{1})+\ldots+\boldsymbol{A}_{d}{{\boldsymbol{B}_{d}}}({s}_{d}) \,\,\, \boldsymbol{s}%
\in\mathbb{R}^{d}_{+}.
\end{equation}
The process $\boldsymbol{B}_{\boldsymbol{A}}$ is a multiparameter  L\'{e}vy process on $\mathbb{R}^{n}$, see Example 4.4 in
\cite{Ba}. We call the multiparameter
 L\'{e}vy process $\boldsymbol{B}_{\boldsymbol{A}}(\boldsymbol{s})$ in
\eqref{Brho1} {\it multiparameter Brownian motion}.

\begin{definition}\label{SubBM}
\label{process} A process $\boldsymbol{Y}$ defined by
\begin{align} \label{GMPP}
\boldsymbol{Y}(t):=\boldsymbol{B}_{\boldsymbol{A}}(\boldsymbol{\SS}(t)),%
\end{align}
where $\BB_{\AA}(\ss)$ is the multiparameter process in \eqref{Brho1} and $\SS(t)$ is a multivariate Sato subordinator independent of $\BB_{\boldsymbol{A}}(\ss)$,
is a Sato subordinated
multiparameter Brownian motion.
\end{definition}
We now provide the characteristic function of $\YY(t)$ in \eqref{GMPP}.
\begin{theorem}\label{SubCh}
The Sato subordinated Brownian motion $\YY(t)$  in \eqref{Brho1}  is an additive pure jump process with time $t$ characteristic function:
\begin{equation}\label{ppbv2}
\hat{\mu}_t( \zz)=exp\{\Gamma(-\alpha)\int_{S^{d-1}}[(\beta(\w)-t^q\langle \log(\hat{\phi}_{\AA}(\zz)), \w\rangle)^{\alpha}-\beta(\w)^{\alpha}]\lambda(d \w)\}, \,\,\, \forall  \zz \in \RR^d,
\end{equation}
where $\log(\hat{\phi}_{\AA}(\zz))=(\log\hat{\phi}_1 (\boldsymbol{z}%
),\ldots,\log\hat{\phi}_d (\boldsymbol{z}))$, $\hat{\phi}_l (\boldsymbol{z}%
)$ is the characteristic function of
$\boldsymbol{A}_{l}\boldsymbol{B}_{l}(1)$, $l=1,\ldots, d$.
\end{theorem}

\begin{proof}
Let $\boldsymbol{A}_i\in \mathcal{M}_{n\times{n_i}}(\RR)$ and let the process
$\BB_{\boldsymbol{A}}$ be defined as in \eqref{Brho1}.
The process $\tilde{\BB}(s_l)= \boldsymbol{A}_l\BB_l(s_l)$ is a $n$-dimensional Brownian motion with parameters $\mmu_{\boldsymbol{A}}=\boldsymbol{A}_l  \mmu_l$ and $\boldsymbol{\Sigma}_l=\boldsymbol{A}_l\boldsymbol{\Sigma}_l \boldsymbol{A}_l^T$.
We have
\begin{equation*}
\BB_{\boldsymbol{A}}(\dd_j)=\BB_{\boldsymbol{A}} (0,\ldots,\underbracket{1}_\text{j-th},\ldots,0)=\boldsymbol{A}_j\BB_j(1).
\end{equation*}
Thus
\begin{align*}
\hat{\phi}_j (\boldsymbol{z}) &=\mathbb{E}[\exp\{i\langle\boldsymbol{A}_j\BB_j(1),\boldsymbol{z}\rangle\}]= \mathbb{E}[\exp\{i\langle\BB_{\boldsymbol{A}}(\dd_j),\boldsymbol{z}\rangle\}]\\
\end{align*}
and
\begin{equation}\label{subSch}
\begin{split}
\hat{\mu}_t(\zz)&=E[\exp\{i\langle \YY(t), \zz\rangle]
=E[E[\exp\{i\langle \BB_{\AA}(\ss), \zz\rangle\}|\SS(t)=\ss]]\\
&=E[\exp\{\langle \log(\hat{\phi}_{\AA}(\zz)), \SS(t)\rangle\}]=\exp\{\psi_{\SS}(t^q \log(\hat{\phi}_{\AA}(\zz))\},\\
\end{split}
\end{equation}
where the second equality follows from Theorem 4.7 in \cite{Ba} and last equality follows because $\SS(t)$ is a Sato process.
Equation \eqref{subSch} with \eqref{ppbv1}, gives \eqref{ppbv2}.
\end{proof}

The unit time distribution of $\YY(t)$ in \eqref{GMPP} is a multivariate Gaussian mixture. The mixing distribution is the unit time distribution of the Sato subordinator $\SS(t)$.
%
%
%
\begin{remark}
Subordination of multiparameter processes has been introduced in \cite{Ba}, where the authors consider the case of a L\'evy subordinator. In \cite{jevtic2019multivariate}, the authors consider the case where the multiparameter process is the multiparameter Brownian motion in \ref{Brho1}.
The Sato subordinated Brownian motion in \eqref{GMPP} has the same unit time distribution of a subordinated multiparameter Brownian motion, as one can see from the unit time characteristic function.
\end{remark}

\section{Application to asset returns modeling}\label{App}

This section proposes an inhomogeneous multivariate additive  model for asset returns based on Sato subordination.
We specify the multivariate Sato subordinator to have a simple but flexible dependence structure.
We build on a dependence structure economically sound that has been proposed  in the L\'evy framework.

\cite{Sem1}  introduced
 factor-based subordinators, i.e. subordinators with a component common to all
assets and an idiosyncratic component. Factor-based subordinators are used to build asset return models, see e.g.  \cite{buchmann2015multivariate} and \cite{guillaume2013alphavg}.  By properly choosing the components, the unit time distribution of factor-based subordinators belongs to the class of multivariate $\ets $ distributions.

\subsection{Sato subordinator}
We start with the introduction of  factor-based $\ets$ distributions.
\begin{proposition}\label{factor}
Let  $\alpha\in (0,1)$. Let us consider a random vector $\SS=(S_1,\ldots, S_n)\sim \mu$, such that
\begin{equation}\label{onefact}
S_j=X_j+a_{j}Z, \,\,\, j=1,\ldots, d,
\end{equation}
where  $Z\sim\ets(\alpha, \beta_Z, \lambda_Z)$ and $X_i\sim\ets(\alpha, \beta_i, \lambda_i)$ are independent.
Then $\SS\sim\ets(\alpha, \beta, \lambda)$, where $\lambda $ has support
 $\text{Supp}(\lambda)=\{\w, \ee_j, j=1,\ldots, d\}$,  $\{\ee_j, j=1,\ldots, d\}$ is the canonical $\RR^d$ basis,  $\w=\frac{\aa}{||\aa||}$ and $\beta:\text{Supp}(\lambda)\rightarrow\RR$, is defined by $\beta(\ee_i)=\beta_i$ and $\beta(\w)=\frac{\beta_Z}{||\aa||}$.
\end{proposition}
\begin{proof}
Let $\beta: Supp(\lambda)\rightarrow \RR_+$, such that $\beta(\w)=\beta_j$ if $\w=\ee_j$ and $\beta(\w)=\frac{\beta_Z}{||\aa||}$ .
The vector
 $\aa Z:=(a_1Z,\dots,a_dZ)^T$ has $\ets$  distribution with parameters $\alpha$, $\beta$ and $\lambda_Z$, where $\lambda_Z$ is a finite measure  with support on the point $\w=\frac{\aa}{||\aa||}$, see Lemma 2.1 in \cite{semeraro2020note}.

 Since $\XX=(X_1, \ldots, X_d)$ has independent components, by Proposition \ref{sumind}, its L\'evy measure has support on
$\{\ee_j, j=1,\ldots, d\}$. Thus
\begin{equation*}
\begin{split}
\nu_{\XX}(E)&=\sum_{j=1}^n\nu_{\XX}(E_{j})=\sum_{i=1}^d\boldsymbol{1}_{E_i}(r \ee_i)\frac{e^{-\beta_ir}}{r^{\alpha+1}}\lambda_i=\sum_{i=1}^d\boldsymbol{1}_{E_i}(r \ee_i)\frac{e^{-\beta({\ee_i})r}}{r^{\alpha+1}}\lambda_{\XX}(\{\ee_i\}).
\end{split}
\end{equation*}
where  $E\in \mathcal{B}(\mathbb{R}^n\setminus \{0\})$,
$E_{j}=E\cap A_j$ and $A_j=\{x\in \mathbb{R}^n: x_k=0, k\neq j,
k=1,...,n\}$.  Thus $\XX\sim \ets(\alpha, \beta, \lambda_{\XX})$ with $\lambda_{\XX}(\{\w\})=\lambda_i$ if $\w=\ee_i$ and $\lambda_{\XX}(\{\w\})=0$ otherwise.

 Since $\aa Z$ and $\XX$ are independent, by Proposition \ref{closure} $\XX+\aa Z\sim \ets(\alpha, \beta, \lambda)$, where $\lambda=\lambda_X+\lambda_Z$ and thus it has support on $\{\w; \ee_j, j=1,\ldots, d\}$.
\end{proof}
We say that $\SS$ in Proposition \ref{factor} has a factor-based  distribution $\mu\sim\ets(\alpha, \beta, \lambda)$.
\begin{corollary}

A factor-based
measure $\mu\sim\ets(\alpha, \beta, \lambda)$   has characteristic function the form
\begin{equation}\label{fbs}
\begin{split}
\hat{\mu}( \zz)=&\prod_{j=1}^d\exp\{\Gamma(-\alpha)[(\beta_j-iz_j)^{\alpha}-\beta_j^{\alpha}]\lambda_j\}\cdot\\
&\cdot\exp\{\Gamma(-\alpha)[(\beta_Z-i(\sum_{k=1}^d a_kz_k))^{\alpha}-\beta_Z^{\alpha}]\lambda_Z\}, \,\,\, \forall  \zz \in \RR^d,\\
\end{split}
\end{equation}
where $\beta_j, \beta_Z\in \RR_+$ and  $\lambda_j, \lambda_Z\in \RR_+$.
\end{corollary}

\begin{proof}
It is sufficient to notice that if $\lambda$ has finite support $\{\w_j, \, j=1,\ldots d\}$, then its characteristic function becomes
\begin{equation*}\label{ppf}
\begin{split}
\hat{\mu}( \zz)
&=\prod_{j=1}^d\exp\{\Gamma(-\alpha)\lambda_j[(\beta_j-i\langle \w_j, \zz \rangle)^{\alpha}-\beta_j^{\alpha}]\}, \,\,\, \forall  \zz \in \RR^d,
\end{split}
\end{equation*}
where $\beta(\w_j)=\beta_j$ and $\lambda(\{\w_j\})=\lambda_j$.
The assert follows by choosing $\lambda$ and $\beta$ as in Proposition \ref{factor}.
\end{proof}
Notice that, if  $\lambda$ has finite support,
the  one-dimensional  marginal distributions $\mu_j$ of  $\mu\sim\ets(\alpha, \beta, \lambda)$  are convolutions of one-dimensional $\ets$ distributions. This means that, by properly choosing the distributions in the  $\ets$ class and their parameters, we can use closure properties of convolution and obtain specified one-dimensional marginal distributions.



Factor-based subordinators with   distributions at unit time that belong to the family of tempered stable distributions  are widely used to introduce dependence in multivariate models in finance.   See, for example, \cite{Sem1},  \cite{LuciSem1} and \cite{buchmann2019weak}. Here we propose a similar construction, but we replace a L\'evy subordinator with a Sato subordinator, to include inhomogeneity of increments.
\begin{definition}\label{Satosub}
A factor-based Sato subordinator $\SS(t)$ is a Sato process such that $\SS(1)$ has the factor-based $\ets$ distribution in Proposition \ref{factor}.
\end{definition}
If $\SS(1)\sim \ets(\alpha, \beta, \lambda)$, the time $t$ characteristic function of the Sato subordinator $\SS(t)$ is
\begin{equation*}
\begin{split}
\hat{\mu}_{t}( \zz)=\hat{\mu}(t^q\zz),
\end{split}
\end{equation*}
where $\hat{\mu}$ is the characteristic function in \eqref{fbs}.   If $\SS(t)$ is the  Sato subordinator associated to $\SS$ we write  $\SS(t)\sim \ets(\alpha, \beta, \lambda, q)$.

\subsection{Factor-based Sato subordinated Brownian motion}

We now introduce two Sato subordinated Brownian motions, with $\ets$  Sato subordinators.
The two models proposed here have the same dependence structure of  the factor-Based subordinated Brownian motions in \cite{LuciSem1}. Actually, the processes we consider have the same unit time distribution of the  factor-based subordinated Brownian motions in \cite{LuciSem1}.
Our aim is to keep the flexibility of their dependence structure and  to have one-dimensional  unit time distributions  in given classes.

We can properly specify the unit time distribution to have one-dimensional margins belonging to classes well suited to model single asset returns. At the same time, we include time inhomogeneity of increments and  time varying correlations.

\begin{definition}\label{dipfb}
 Let $B_j(t), j=1,\ldots, d$ be independent Brownian motions with drift $\mu_j$ and diffusion $\sigma_j$.
Let $\boldsymbol{B}^{\rho}(t)$ be a correlated
$d$-dimensional Brownian motion, with correlations $\rho_{ij}$, marginal
drifts $\boldsymbol{\mu}^{\rho}_j=\mu_{j}\alpha_{j}$ and
diffusion matrix $\Sigma^{\rho}:=( \rho_{ij}\sigma_{i}\sigma_{j}\sqrt
{{\alpha_{i}}}\sqrt{{\alpha_{j}}})_{ij}$, $i,j=1,\ldots d$.
The $\mathbb{R}^{d}$-valued subordinated process $\boldsymbol{Y}%
=\{\boldsymbol{Y}^{\rho}(t),t>0\}$ defined by%

\begin{equation}
\boldsymbol{Y}^{\rho}(t)=\left(
\begin{array}
[c]{c}%
B_{1}(X_{1}(t))+B_{1}^{\rho}(Z(t))\\
....\\
B_{d}(X_{d}(t))+B_{d}^{\rho}(Z(t))
\end{array}
\right)  , \label{abgp}%
\end{equation}
where $X_{j}(t)$ and $Z(t)$ are the independent Sato subordinators with unit time distribution of $X_j$ and $Z$  in Proposition \ref{factor}, independent from
$\boldsymbol{B}(t)$ and $\boldsymbol{B}^{\rho}(t)$ is a $\rho$-factor-based Sato subordinated
Brownian motion.
\end{definition}
The following proposition can be proved as Proposition 3.1  in \cite{jevtic2019multivariate}.
\begin{proposition}
Let $\boldsymbol{Y}^{\rho}(t)$ be a  $\rho$-factor-based Sato subordinated
Brownian motion in \eqref{abgp}. Then $\boldsymbol{Y}^{\rho}(t)$
belong to the class of Sato subordinated Brownian motions in Definition \ref{SubBM}.
\end{proposition}

\begin{proof}
Let us consider the  independent  Brownian motions $B_i(t), \, i=1,\ldots d$ and $\BB^\rho(t)$ \ref{abgp} and let
\begin{equation*}
\BB_{\boldsymbol{A}}(\ss)=\sum_{i=1}^dA_iB_i(s_i)+\AA_{d+1}\BB^{\rho}(s_{d+1}),
\end{equation*}
where $A_i\in \mathcal{M}(d\times 1)$, $i=1,\ldots, d$ and $\boldsymbol{A}\in \mathcal{M}_{d\times (d)}$ such that
$A_i=(0,\ldots,1,\ldots, 0)$, $i_1,\ldots, d$
and $\AA_{d+1}=\boldsymbol{I}_d$. Let now $\SS(t)=(X_1(t),\ldots, X_d(t), Z(t))$ be and $\ets$ Sato subordinator with independent components. Let $\YY^{\rho}(t)$ be the process in \eqref{abgp}
we have  $\YY(t)^{\rho}=\BB_{\AA}(\SS(t))$.
\end{proof}

\begin{corollary}\label{chfb}
If $\SS(t)\sim \ets(\alpha, \beta, \lambda, q)$ is the Sato subordinator in \eqref{onefact} the subordinated process $\YY(t)$ in \eqref{abgp} has characteristic function
\begin{equation}\label{BMCh}
\begin{split}
\hat{\mu}_t( \zz)=&\prod_{j=1}^d \exp\{-2\sqrt{\pi}[(\beta_j-  t^q (i\mu_jz_j-\frac{1}{2}\sigma^2z_j^2 )^{\alpha}-\beta_j^{\alpha}]\lambda_j\}\cdot\\
&\exp\{-2\sqrt{\pi}[(\beta_Z-  t^q (i\boldsymbol{u}^{T}\boldsymbol{\mu}^{\rho}%
-\frac{1}{2}\boldsymbol{z}^{T}\boldsymbol{\Sigma}^{\rho}\boldsymbol{z})^{\alpha}-\beta_Z^{\alpha}]\lambda_Z\}, \,\,\, \forall  \zz \in \RR^d,
\end{split}
\end{equation}
where $\beta_j, \beta_Z, \lambda_j$ and $\lambda_Z$ are the parameters in Proposition \ref{factor}.
\end{corollary}

\begin{corollary}\label{cor}
If the Brownian motions in Definition \ref{dipfb} have zero drift  the subordinated process $\YY^{\rho}(t)$ in \eqref{abgp}  is a Sato process.
\end{corollary}

\begin{proof}
Since $\YY^{\rho}(t)=\BB_{A}(\SS(t))$, it is the sum of the  independent subordinated Brownian motions $B_i(S_i(t)),\, i=1,\ldots, d$ and $\BB^{\rho}(S_{d+1}(t))$. It is sufficient to observe that if  Brownian motions in Definition \ref{dipfb} have zero drift  the unit time distribution of $B_i(S_i(t)),\, i=1,\ldots, d$ and $\BB^{\rho}(S_{d+1}(t))$ are self-decomposable (\cite{takano1989mixtures}), therefore  the unit time distribution of  and  $\YY^{\rho}(t)$ is. Furthermore, $t^{2q}\log(\hat{\phi}_{\AA}(\zz))=\log(\hat{\phi}_{\AA}(t^q\zz))$.
\end{proof}
\begin{proposition}
If all the parameters $\rho_{ij}$ in \eqref{abgp}  collapse to 0 across
different components, i.e. $\rho_{ij}=0$, for $i\neq j,$ $\rho_{ij}=1$, for
$i=j,$ the factor-based  Sato subordinated Brownian motion $\YY^{\rho}(t)$ in \eqref{abgp} has the simpler form:
\begin{equation}\label{sbS}
\YY(t)=\BB(\SS(t))=(B_1(S_1(t)), \ldots, B_d(S_d(t))),
\end{equation}
where $\SS(t)$ is the factor-based Sato subordinator in Definition  \ref{Satosub} and $\BB(\ss)$ is a $d$-dimensional multiparameter  Brownian motion with independent components $B_j(s_j)$ that have mean $\mu_j$ and diffusion $\sigma_j$.
\end{proposition}
\begin{proof}
By substituting $\rho_{ij}=0$ in \eqref{abgp} we The characteristic function of $\YY^{\rho}(t)$ becomes
\begin{equation*}\label{BMChI}
\begin{split}
\hat{\mu}_t( \zz)=
&\prod_{j=1}^d \exp\{-2\sqrt{\pi}[(\beta_j-  t^q (i\mu_jz_j-\frac{1}{2}\sigma^2z_j^2 )^{\alpha}-\beta_j^{\alpha}]\lambda_j\}\cdot\\
&\cdot\exp\{\Gamma(-\alpha)[(\beta_Z-it^q(\sum_{k=1}^d a_k(\mu_kz_k-\frac{1}{2}\sigma^2z_k^2))^{\alpha}-\beta_Z^{\alpha}]\lambda_Z\}\\
&=\exp\{\psi_{\SS}(t^q \log(\hat{\phi}_{\II}(\zz))\} , \,\,\, \forall  \zz \in \RR^d,\\
\end{split}
\end{equation*}
where $\hat{\phi}_{\II}(\zz)$ is the characteristic function of the multivariate Brownian motion in \eqref{sbS}. From Theorem \ref{SubCh} it follows that  $\hat{\mu}_t(\zz)=\exp\{\psi_{\SS}(t^q \log(\hat{\phi}_{\II}(\zz))\}$   is the characteristic function of $\YY(t)$ in \eqref{sbS}.
\end{proof}
A very useful result for applications and calibration is the following, that can be easily proved using  characteristic functions, see Theorem 5.1, \cite{luciano2010generalized}.
\begin{proposition}\label{margp}
The processes in \eqref{abgp} and in \eqref{sbS} have the same one-dimensional marginal processes in law, that are one-dimensional Sato subordinated Brownian motions.
\end{proposition}
As a consequence the processes \eqref{abgp} and \eqref{sbS} clearly have one-dimensional marginal processes that are themselves Sato-subordinated Brownian motions.

 Linear correlation is usually used to calibrate the dependence structure in multivariate models and it is usually constant over time, although this is not a realistic assumption, see e.g. \cite{toth2006increasing}, \cite{teng2016dynamic} and \cite{lundin1998correlation}.
We now show that  linear correlations of these processes  changes over time and have a simple analytical formula.
Standard computations give mean $E[Y^{\rho}_j(t)]$, variance $V[Y^{\rho}_j(t)]$ and correlation $\rho _{\boldsymbol{Y}^{\rho}(t)}(h,j)$ of $\YY^{\rho}(t)$:
\begin{equation*}
\begin{split}
E[Y^{\rho}_j(t)]&=\mu t^qE[S_j];\,\,\,
V[Y^{\rho}_j(t)]=
\sigma^2t^qE[S_j]+\mu^2t^{2q}V[S_j].
\end{split}
\end{equation*}
\begin{equation}\label{Lcorr}
\rho _{\boldsymbol{Y}^{\rho}(t)}(h,j)=\frac{\rho_{hj}\sigma_h\sigma_j\sqrt{a_h}\sqrt{a_j}t^{q}E[Z]+\mu _{j}\mu _{h}a_ja_h t^{2q}V({Z})}{\sqrt{(\sigma_j^2 t^qE[S_j]+\mu_j^2 t^{2q}V[S_j])(\sigma_h^2 t^qE[S_h]+\mu_h^2 t^{2q}V[S_h])}}
\end{equation}
The linear correlation coefficients of $\YY(t)$ in \eqref{sbS} are obtained by assuming $\rho_{ij}=0$.
%
Linear correlations are functions of time. Furthermore,
\begin{equation*}
\lim_{t\rightarrow \infty }\rho _{\boldsymbol{Y}^{\rho}(t)}(l,j)=\rho _{\boldsymbol{S}}(l,j)
\end{equation*}
and
\begin{equation*}
\lim_{t\rightarrow 0}\rho _{\boldsymbol{Y}^{\rho}(t)}(l,j)=\frac{\rho_{hj}\sqrt{a_1a_j}E{[Z]}}{\sqrt{E[S_l]E[S_j]}}.
\end{equation*}
Since $E[S_j]=E[X_j]+a_jE[Z]$, $\lim_{t\rightarrow 0}\rho _{\boldsymbol{Y}^{\rho}(t)}(l,j)= 1$ if we have the limit values $\rho_{ij}=1$ and $E[X_j]=0$, $i,j=1,\ldots, d$.

The unit time distribution of  $\YY^{\rho}(t)$  in \eqref{abgp}  is a multivariate Gaussian mixture. This is the same unit time distributions of the $\rho\alpha$-models L\'evy processes in \cite{LuciSem1} and of the $\rho\alpha$-Sato process in \cite{marena2018multivariate}.

In  \cite{marena2018multivariate} the authors rely on the fact that under proper conditions the  one-dimensional marginal distributions of   $\YY^{\rho}(1)$  are self-decomposable and introduced a multivariate process with Sato one-dimensional dynamics, the $\rho\alpha$-Sato process. This process, as discussed in the introduction is not a multivariate Sato process. Furthermore, although it  introduces time inhomogeneity for single assets, it still has correlations constant over time.
The $\rho\alpha$-model and the $\rho\alpha$-Sato model are recalled in Appendix \ref{Anig}.

In the next section, we specify a distribution for the factor-based Sato subordinated Brownian motion and consider a numerical example to exhibit  the dynamics of correlations.

\subsection{Normal inverse Gaussian case}

\label{NIG}

We specify the factor-based Sato subordinator to have multivariate inverse Gaussian $\ets$ distribution and we choose the parameters to have the same unit time distribution of the factor-based $\rho\alpha$-NIG in \cite{LuciSem1} and of the factor-based multivariate  Sato-NIG in \cite{marena2018multivariate}.
To this purpose we choose the distribution of the subordinator factors according to a standard  parametrization of the univariate IG distribution. An inverse Gaussian (IG) distribution with parameters $%
(a,b)$ is a L\'{e}vy process with the following characteristic function%
\begin{equation*}
\hat{\mu}_{IG}(u)=\exp \left( -a\left( \sqrt{b^{2}-2iu}-b\right) \right) .
\end{equation*}
Let now $\SS=(S_1,\ldots, S_d)$ as in \eqref{onefact} with
\begin{equation}\label{sub}
X_{j}\sim IG\left(  1-{a}\sqrt{a_{j}},\frac{1}{\sqrt{a_{j}}}\right)
,\,j=1,...,n\quad\text{and}\quad Z\sim IG(a,1),
\end{equation}
where
$
0<a<\frac{1}{\sqrt{a_{j}}},\quad j=1,...,n \label{nig_constraint_sub}%
$.
Let now $\gamma_{j},\beta_{j},\delta_{j}$ be such that%
$
\gamma_{j}>0$,$-\gamma_{j}<\beta_j<\gamma_{j}$, $\delta_{j}>0$;
further, let%
$
\frac{1}{\sqrt{a_{j}}}=\delta_{j}\sqrt{\gamma_{j}^{2}-\beta_{j}^{2}}.
$

Let now $\SS(t)$ be the factor-based Sato subordinator in Definition \ref{Satosub} with unit time distribution of $\SS$, its time $t$ characteristic   function $\hat{\mu}_{\SS(t)}(\zz)$ has the form
\begin{equation*}
\begin{split}
\hat{\mu}_{\SS(t)}( \zz)=&\prod_{j=1}^d\exp\{(1-a\sqrt{a_j})[\sqrt{(\frac{1}{a_j}-2it^qz_j)}-\frac{1}{\sqrt{a_j}}]\}\cdot\\
&\cdot\exp\{-a[\sqrt{(1-2i(t^q\sum_{k=1}^d a_kz_k))}-1]\}, \,\,\, \forall  \zz \in \RR^d.\\
\end{split}
\end{equation*}
By construction the unit time distribution of $\SS(t)$ has one-dimensional $IG(1, \frac{1}{\sqrt{a_j}})$ marginal distributions that are self-decomposable. Therefore the processes $S_j(t)$ are one-dimensional Sato processes.
\begin{definition}\label{S-IG}
Let $\boldsymbol{Y}^{\rho}(t)$ be the process  defined in \eqref{abgp}. If $X_j(t),\,\,i=1,\ldots, d$ and $Z(t)$ are Sato subordinators with unit time distributions in \eqref{sub} and if we set $\mu_{j}=\beta_{j}\delta_{j}^{2}$ and $\sigma_{j}=\delta_{j}$ the process $\YY^{\rho}(t)$ is called factor-based Sato-IG subordinated Brownian motion.
\end{definition}

From Proposition \ref{margp} the process $\YY^{\rho}(t)$  in Definition \ref{S-IG} has the following one marginal
processes
\begin{equation}\label{NigMar}
Y^{\rho}_{j}(t)=\beta_{j}\delta_{j}^{2}S_{j}(t)+\delta
_{j}W(S_{j}(t)),
\end{equation}
where $W(t)$ is a standard Brownian motion and equality is in distribution. Therefore the unit time distributions of $\YY^{\rho}_j(t)$ in \eqref{NigMar} are  normal inverse Gaussian distributions with parameters $(\gamma_j, \beta_j, \delta_j)$, that are self-decomposable. The one-dimensional processes $Y^{\rho}_j(t)$, $j=1,\ldots, d$ are Sato subordinated Brownian motions. Therefore in a multivariate asset model the process provides time varying correlations and time inhomogeneity also for single assets.
The process $\YY^{\rho}(t)$ has a total of $2+3n+\frac
{n(n-1)}{2}$ parameters: $a$ and $q$ are common parameters; $\gamma_{j},\beta
_{j},\delta_{j},\,j=1,...,n$ are marginal parameters and $\rho_{ij}%
,i,j=1,...,n,$ are the $\boldsymbol{B}^{\rho}$ correlations.

From \eqref{BMCh}, the  time $t$ characteristic function $\hat{\mu}^{\rho}_t$ of $\YY^{\rho}(t)$,  is%
\begin{equation*}%
\begin{split}
\hat{\mu}_{t}^{\rho}(\boldsymbol{u})  &  =-\prod_{j=1}%
^{d}\exp\left\{  (1-\frac{a}{\zeta_{j}})\left(  \sqrt{(\zeta_{j}^{2}-2t^q(i\beta_{j}\delta_{j}^{2}%
u_{j}+\frac{1}{2}\delta_{j}^{2}u_{j}^{2}))}-\zeta_{j}\right)
\right. \\
&  \left.  -a\left(  \sqrt{1-2t^{q}(i\boldsymbol{u}^{T}\boldsymbol{\mu}^{\rho}%
-\frac{1}{2}\boldsymbol{u}^{T}\boldsymbol{\Sigma}^{\rho}\boldsymbol{u}%
)}-1\right)  \right\}
\end{split}
\label{raNIGch}%
\end{equation*}
where $\zeta_{j}=\delta_{j}\sqrt{\gamma_{j}^{2}-\beta_{j}^{2}}.$ The linear correlations
  are given in \eqref{Lcorr} with $\sigma_j=\delta_j$, $\mu_j=\beta\delta_j^2$, $\sqrt{a_j}=\frac{1}{\zeta_j}$, $E[Z]=a$ and $V[Z]=a$ where $a\leq min\{\frac{1}{\sqrt{a_j}}\}$.
%
%
The $\rho\alpha$-NIG in \cite{LuciSem1} is the L\'evy process with unit time characteristic function $\hat{\mu}^{\rho}_1(\zz)$. A $\rho\alpha$-NIG process $\YY^L(t)$ has correlations constant over time and $\rho_{\YY^L(t)}(i,j)=\rho_{\boldsymbol{Y}^{\rho}(1)}(i,j)$, for each $t>0$.
The marginal   Sato-NIG  process $\XX(t)$ introduced in \cite{marena2018multivariate} has marginal Sato processes with different Sato exponents. The $\rho\alpha$-NIG and the Sato-NIG processes are both  recalled in Appendix \ref{Anig}, see equations \eqref{abgp2} and \eqref{Sabgp}, respectively.

We consider here the subcase with a common Sato exponent, since  the Sato exponent do not affect correlations of the Sato-NIG process. In this subcase the Sato-NIG process  has characteristic function given by:
\begin{equation*}
\hat{\mu}_{\XX}(\zz)=\hat{\mu}_{1}^{\rho}(t^q\zz),
\end{equation*}
where $\hat{\mu}_{1}^{\rho}(\zz)$ is the unit time  characteristic function in \eqref{raNIGch}. Although the Sato-NIG has time inhomogeneous increments it has
correlations  constant over time. The Sato-NIG correlations are $\rho_{\XX(t)}(i,j)=\rho_{\boldsymbol{Y}^{\rho}(1)}(i,j)$, for each $t>0$.


Notice that, in the symmetric case, the factor-based S-IG subordinated Brownian motion in \eqref{S-IG} and the   Sato-NIG  process have the same law at any time $t$ (by changing the parametrization of the Sato exponent), according to Corollary \ref{cor}. This implies that we have time varying correlations  in the asymmetric case.
The empirical investigation of the fit of this model on real data is out of the aim of this paper. However,
the next section provides a numerical example to show possible  dynamics for  correlations, using realistic parameters.
\subsection{Numerical illustration}

Define a bidimensional price process, $\boldsymbol{S}=\{\boldsymbol{S}%
(t),\,t\geq 0\}$, by%
\begin{equation*}
\boldsymbol{S}(t)=\boldsymbol{S}(0)\exp (\boldsymbol{c}t+\boldsymbol{Y}%
(t)),\,\boldsymbol{c}\in \mathbb{R}^{n},  \label{drift}
\end{equation*}%
where $\boldsymbol{c}$ is the drift term (equivalently, $S_{j}(t)=S_{j}(0)%
\exp (c_{j}t+Y_{j}(t)),\,t\geq 0,j=1,2.$).

Both the Sato-NIG process and the $\rho\alpha$-NIG L\'evy processes are used to model the return process $\YY(t)$.
These two processes, assuming the same (unit time) marginal parameters,   have the same correlations that are constant over time and are equal to unit time correlation of the factor-based S-IG Brownian motion.
In  \cite{luciano2016dependence} the authors performed an estimate on real data of the $\rho\alpha$-NIG model under
the historical measure.
They considered daily log-returns on MSCI US Investable Market Indices from
January 2, 2009 to May 31st, 2013.  They had 10 indices, with a total of 1109 observations.
The calibration was  performed in two
steps.
The first consisted in fitting the marginal parameters from marginal
return data; the second in selecting the common parameters by matching the
historical return correlation matrix. The marginal return
parameters were calibrates by maximum likelihood (MLE).

We use their estimates of the marginal parameters $\gamma_j, \beta_j, \delta_j$ of the first two indices (consumer discretionary (CD) and consumer staples (CS))  and move the correlation parameters and the Sato exponent to exhibit possible correlation  dynamics  over time in a bivariate setting.

The   estimated marginal parameters $\gamma_j, \beta_j, \delta_j$  are taken from \cite{luciano2016dependence}, Table 2 and they are  reported  in Table \ref{table_param}.

Notice  that the Sato exponent $q$, that drives the correlation dynamics,  is both a common and  a marginal parameter.

\begin{table}[tbp]
\centering
{\footnotesize \centering
\centering
\begin{tabular}{r|cccc}
\hline
Index &  $\gamma$ & $\beta$ & $\delta$ \\
\hline
CD &  51.7708 & -5.0441 & 0.0112 \\
CS &  108.3392 & -12.8277 & 0.0076 \\
EN &  54.9486 & -6.0927 & 0.0155
\\
FN &  22.7119 & -1.7045 & 0.0113 \\
HC &  82.5935 & -13.7078 & 0.0090  \\
IN &  45.0711 & -5.2494 & 0.0115  \\
IT &  57.3094 & -4.3395 & 0.0114  \\
MT &  54.3748 & -7.4708 & 0.0159 \\
TC &  81.6045 & -12.0085 & 0.0101 \\
UT &  97.9514 & -7.5590 & 0.0098 \\
\hline
\end{tabular}
}
\caption{Maximum likelihood estimates of marginal NIG return distributions for the $\protect\rho \protect\alpha $NIG
 distribution from \cite{luciano2016dependence}. MSCI US Investable Market Index from January 2, 2009 to May
31st, 2013.}
\label{table_param}
\end{table}


We show the correlation dynamics over time of the two indices CD and CS for different values of the subordinator common parameter $a$, the Brownian motion correlation $\rho$ and the the Sato exponent $q$. 
 We change one parameter at the time.
We set two values for $a$: $a=\frac{a_{max}}{2}; a_{max}-0.01$ (the case $a=0$ is independence).
We set three values for $\rho$: $\rho=-0.5; \rho=0.5; \rho=0.99$.
We set three values for $q$: $q=0.5; q=1; q=1.5$.

Figure \ref{corr-dyn} illustrates  different dynamics of the time varying correlation if the return process is modelled by  $\YY^{\rho}(t)$ in \eqref{S-IG}.  Figure \ref{corr-dynrho0} exhibits four  different  correlation dynamics for the subcase with independent Brownian motions, i.e. $\rho=0$ and the return process becomes $\YY(t)$   in \eqref{sbS}. We consider t=1109 as time horizon, since it  is the number of daily observation in the sample. We exhibit two set of parameters for two pair of indices because all the other cases have similar dynamics. We can see that  correlation changes in time, it can cover a wide range and it can be increasing or decreasing over time for both the processes $\YY^{\rho}(t)$ and its subcase $\YY(t)$.

\begin{figure}
\centering
\subfloat[CASE1]{\includegraphics[scale=0.15]{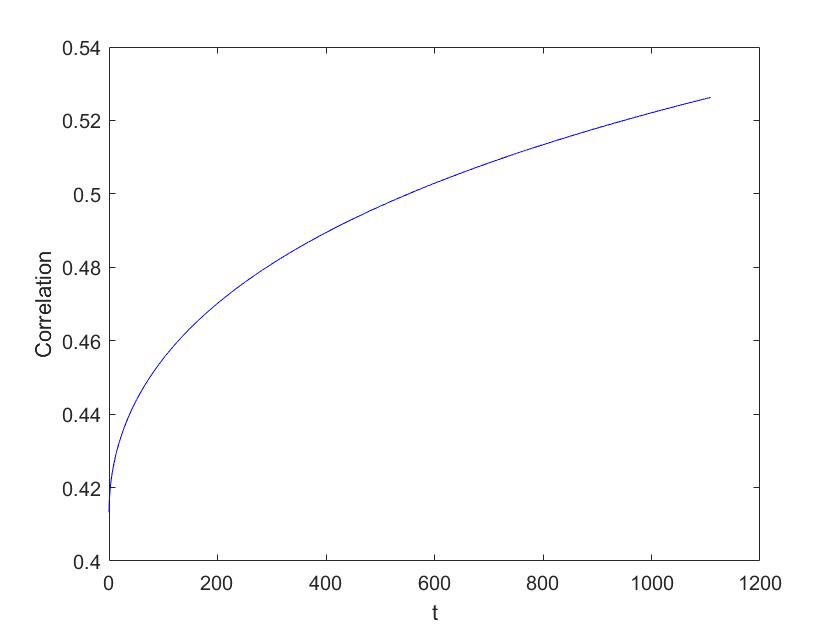}}
\subfloat[CASE2]{\includegraphics[scale=0.15]{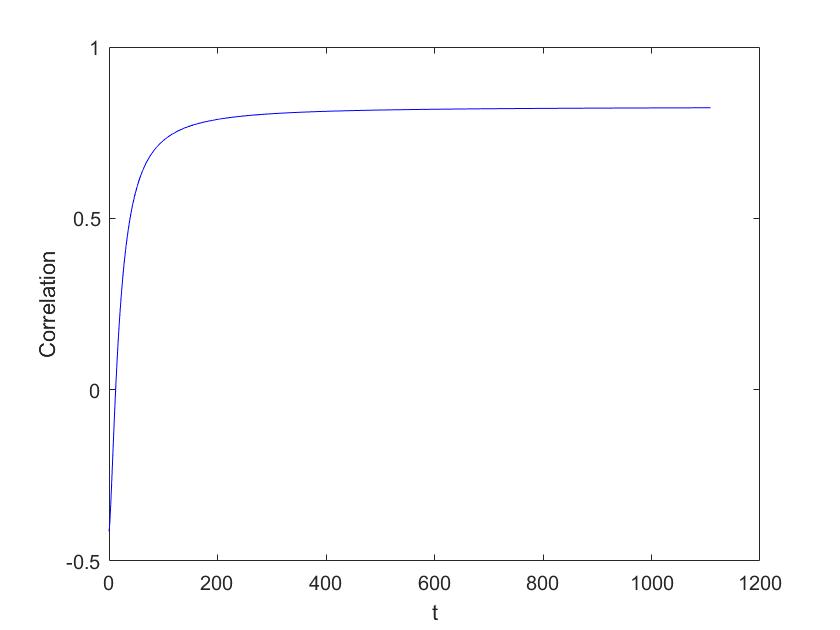}}\newline
\subfloat[CASE3]{\includegraphics[scale=0.15]{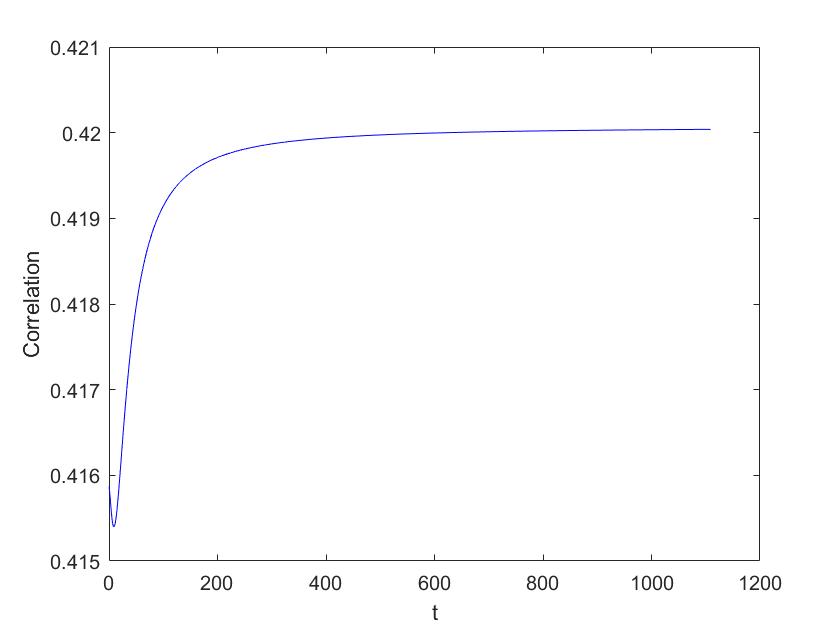}}
\subfloat[CASE4]{\includegraphics[scale=0.15]{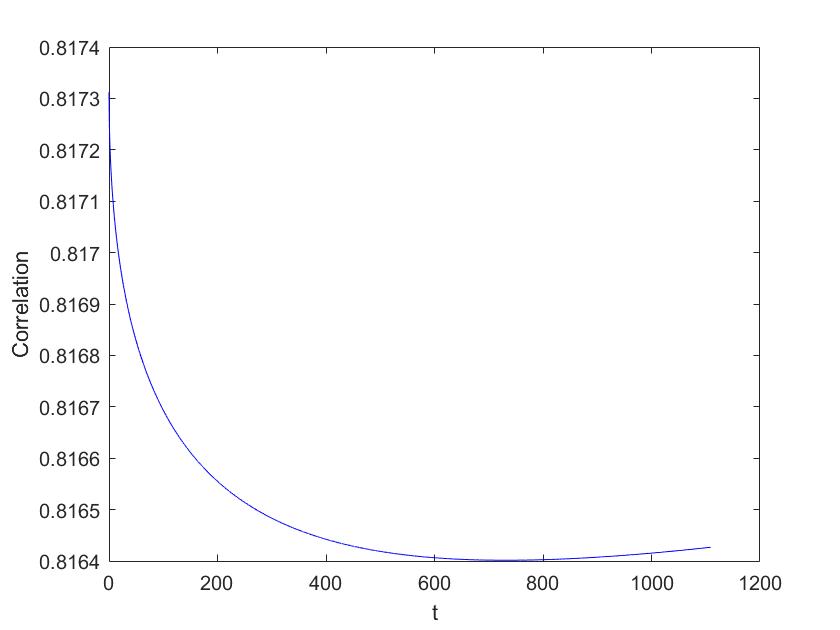}}\newline
\caption{Correlation dynamics of indicies CD, CS for the four choices of the common  parameters:
CASE1: $a=0.5671$, $\rho=0.5$, $q=0.5$; CASE2: $a=0.5671$, $\rho=-0.5$, $q=1.5$
 CASE3: $a=0.2885$, $\rho=0.99$, $q=1.5$; CASE4: $a=amax$, $\rho=0.99$, $q=0.5$ }
\label{corr-dyn}
\end{figure}

Tables \ref{table_corr} and \ref{table_corr2} provides the limit correlations at time $0$ and $t\rightarrow \infty$ and the unit time correlation. Depending on $a, \rho$ and $q$, with the same (unit time)  marginal parameters, we have different ranges spanned from correlation over time.

\begin{table}[h!]
\centering
{\footnotesize \centering
\centering
\begin{tabular}{r|cccc}
\hline
Pair &  $\lim_{t\rightarrow 0}\rho _{\boldsymbol{Y}^{\rho}(t)}(l,j)$ & $\lim_{t\rightarrow \infty}\rho _{\boldsymbol{Y}^{\rho}(t)}(l,j)=\rho_{\SS}$ & $\rho_{\YY(1)}$ \\
\hline
CASE1   & 0.4128 & 0.8256 & 0.4175 \\ \hline
CASE2   &     -0.4128
 &    0.8256&   -0.3984\\ \hline
CASE3   &   0.4159& 0.4201& 0.4158\\ \hline
CASE4&  0.8173&   0.8256 &  0.8172\\ \hline
\end{tabular}
}
\caption{Correlations bound for the pair CD, CS and parameters $a$, $\rho$ and $q$ corresponding to the fours cases in Figure \ref{corr-dyn}.}
\label{table_corr}
\end{table}

\begin{figure}[h!]
\centering
\subfloat[CASE01]{\includegraphics[scale=0.15]{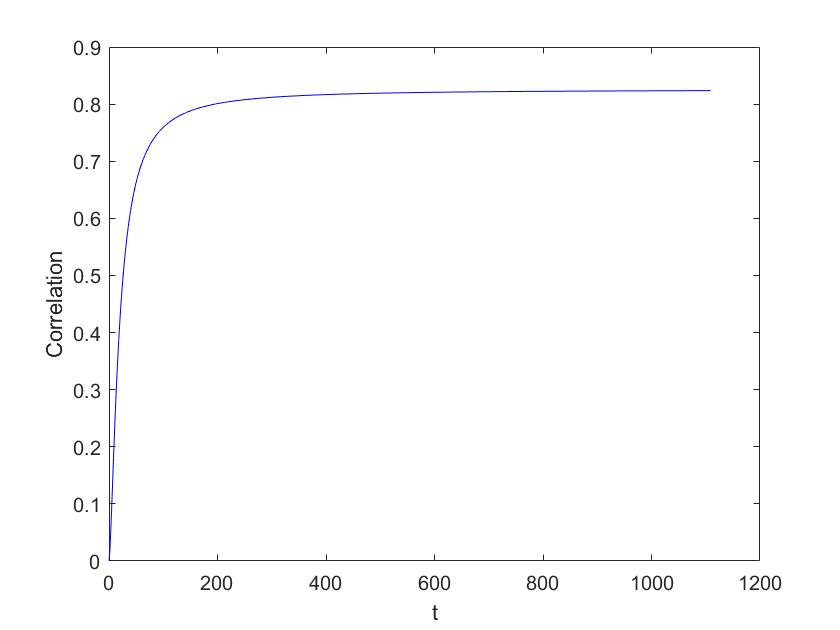}}
\subfloat[CASE02]{\includegraphics[scale=0.15]{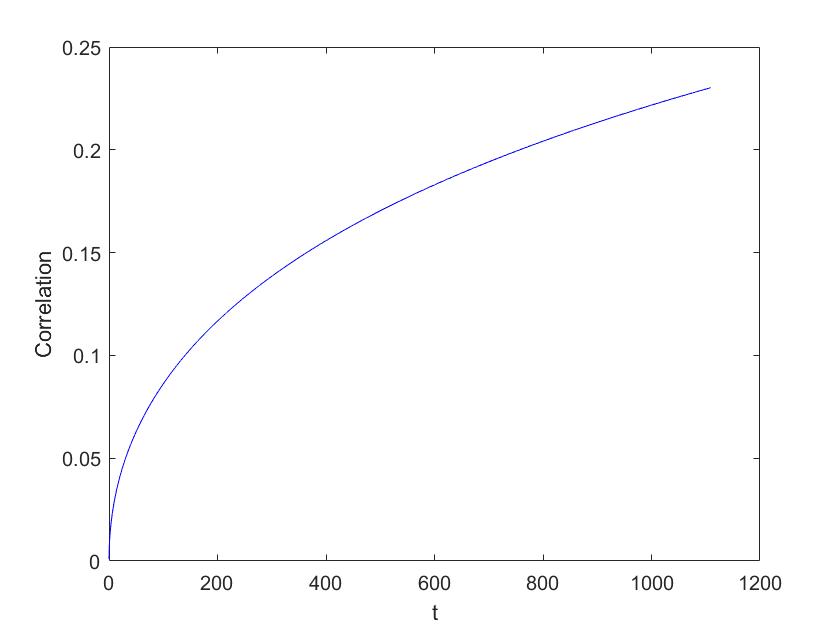}}\newline
\subfloat[CASE03]{\includegraphics[scale=0.15]{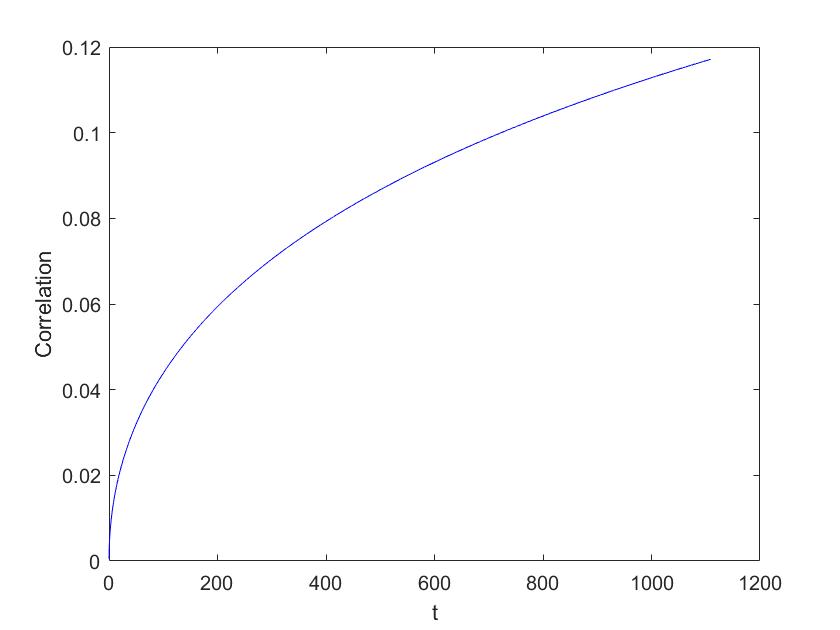}}
\subfloat[CASE04]{\includegraphics[scale=0.15]{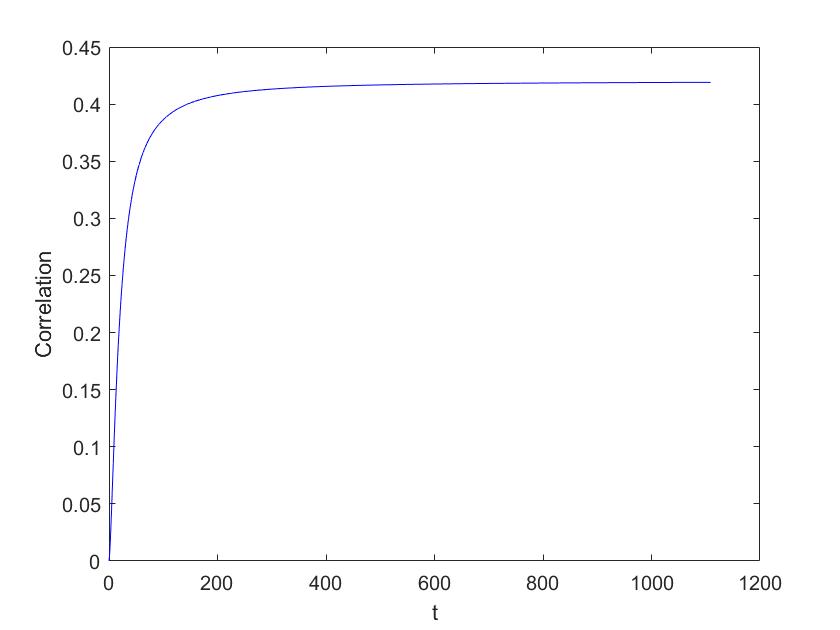}}\newline
\caption{Correlation dynamics of indicies CD, CS for independent Brownian motions and the four choices of the remaining common  parameters:
CASE01: $a=0.5671$, $q=1.5$; CASE02: $a=0.5671$, $q=0.5$
 CASE03: $a=0.2885$, $q=0.5$; CASE04: $a=0.5671$, $q=1.5$ }
\label{corr-dynrho0}
\end{figure}

\begin{table}[h!]
\centering
{\footnotesize \centering
\centering
\begin{tabular}{r|cccc}
\hline
Pair &  $\lim_{t\rightarrow 0}\rho _{\boldsymbol{Y}^{\rho}(t)}(l,j)$ & $\lim_{t\rightarrow \infty}\rho _{\boldsymbol{Y}^{\rho}(t)}(l,j)=\rho_{\SS}$ & $\rho_{\YY(1)}$ \\
\hline
CASE01   & 0 & 0.8256 & 0.0095 \\ \hline
CASE02   & 0 & 0.8256 & 0.0095 \\ \hline
CASE03   &   0& 0.4201& 0.0048\\ \hline
CASE03 &   0& 0.4201& 0.0048\\ \hline
\end{tabular}
}
\caption{Correlations bound for the pair CD, CS and parameters $\rho=0$,  $a$ and $q$ corresponding to the fours cases in Figure \ref{corr-dynrho0}.}
\label{table_corr2}
\end{table}
\section{Conclusion}\label{conc}

We have introduced and characterized  a self-decomposable  class of multivariate exponential tempered distributions and the associated multivariate Sato subordinators.  Then we have constructed a multivariate additive process able to incorporate time inhomogeneity by additive subordination of a multivariate Brownian motion.
The main contribution of Sato subordination  is to provide time varying correlations, while remaining parsimonious in the number of parameters and providing a good fit on financial data.
Although the empirical investigation of Sato subordination is beyond the purpose of this paper with a toy example we have shown  that the model is capable to span a wide range of correlation over time and to exhibit different correlation dynamics for given unit time marginal parameters.

The next steps of our research are in two directions. The first direction is to empirically study the fit of factor-based Sato subordinated Brownian motions on financial data and study their moments term structure.

The second direction is to weaken the assumption of a Sato exponent common to all assets.
In this work, we built on self-decomposability of multivariate $\ets$ distributions and considered a multivariate  Sato subordinator with scaling parameter $q$ common to all the one-dimensional processes.
 Since the factor-based S-IG subordinator  has self-decomposable  marginal distributions, a natural extension is to allow each marginal process to have its own Sato exponent as for the Sato-NIG in  \cite{marena2018multivariate}, but preserving time varying correlation.  To this aim a further step  will be to study   operator self-decomposability of multivariate $\ets$-distributions and the  associated Sato subordinators. The study of multivariate Sato  subordination of more general Markov processes is also in the agenda of our future research.
\section*{Acknowledgements}

The author wishes to thank Marina Marena and Andrea Romeo for the helpful discussions she had with them. The author also gratefully acknowledge financial support from the Italian
Ministry of Education, University and Research (MIUR), "Dipartimenti di
Eccellenza" grant 2018-2022.

\appendix
\appendixpage
\section{Self-decomposability and Sato processes}\label{Aself}

The L\'evy measure of a self-decomposable distribution $\mu$  is of the form  (\cite{Sa}, Theorem 15.10):

\begin{equation}\label{sd1}
\nu(E)=\int_{S^{d-1}}\int_0^{\infty}\boldsymbol{1}_E(r\w)\frac{k_{\w}(r)}{r}dr\lambda(d\w)
\end{equation}
where $\lambda$ is a finite measure on $S^{d-1}$, $k_{\w}(r)\geq 0$ is measurable in $\w$, decreasing in $r>0$.
The L\'evy measure of a  self-decomposable distribution on $\RR$ has the form: 
\begin{equation*}
\begin{split}
\nu(E)&=\int_{\RR}\boldsymbol{1}_E(x)\frac{k(x)}{||x||}dx\,\
\end{split}
\end{equation*}
where $k(x)\geq 0$ is increasing on $(-\infty, 0)$ and decreasing on $(0, \infty)$.
Therefore a multivariate measure $\mu\in L(\RR^d)$ is self-decomposable if and only if its radial component is self-decomposable.
The characteristic function of a $d$-dimensional self-decomposable distribution $\mu$  in cartesian coordinates has  the form
\begin{equation*}\label{sdcart}
\hmu(\zz)=\exp\{i\boldsymbol{\gamma}\cdot \zz+\int_{\RR^d}(e^{i\langle\zz,\xx\rangle}-1-i\langle\zz,\xx\rangle\boldsymbol{1}_{||\xx||\leq 1}(\xx))\frac{h(\xx)}{||\xx||}\tilde{\lambda}(d\xx)\},
\end{equation*}
where
\begin{equation}\label{tildelambda}
\tilde{\lambda}(E)=\int_{S^{d-1}}\int_0^{\infty} \boldsymbol{1}_E(r\w)dr\lambda(d\w).
\end{equation}
and  $h(\xx)=k_{\frac{\xx}{||\xx||}}(||\xx||)$.
Thus
\begin{equation*}
\begin{split}
\nu(E)&=\int_{S^{d-1}}\int_0^{\infty}\boldsymbol{1}_E(r\w)\frac{k_{\w}(r)}{r}dr\lambda(d\w)\\
&=\int_{\RR^d}\boldsymbol{1}_E(\xx)\frac{k_{\frac{x}{||\xx||}}(||\xx||)}{||\xx||}\tilde{\lambda}(d\xx).
\end{split}
\end{equation*}
If  $\lambda$ is the measure $\sigma$ induced on $S^{d-1}$ by  the Lebesgue measure on $\RR^d$ we have
\begin{equation}\label{muLeb}
\hmu(\zz)=\exp\{i\boldsymbol{\gamma}\cdot \zz+\int_{\RR^d}(e^{i\langle\zz,\xx\rangle}-1-i\langle\zz,\xx\rangle\boldsymbol{1}_{||\xx||\leq 1})\frac{h(\xx)}{||\xx||^d}d\xx\},
\end{equation}

A distribution is self-decomposable if and only if for any fixed $q>0$ it is the distribution of $\YY(1)$ for some additive process $\{\YY(t),\,  t>0\}$ which is q-self-similar (see \cite{Sa} as a standard reference for self-similar processes 
). A process $\YY(t)$ is $q$-self-similar if for each $a>0$
\begin{equation*}
\boldsymbol{Y}(a t)\overset{\mathcal{L}}{=}a^q\boldsymbol{Y}(t).
\end{equation*}
where $\overset{\mathcal{L}}{=}$ denotes  equality in distribution of processes.
The probability law of a Sato process at time $t$ is obtained by scaling a
self-decomposable law $\mu$ (see
\cite{carr2007self}).  If $\YY(t)$ is a Sato process, we have :
\begin{equation*}
\boldsymbol{Y}(t)\overset{\mathcal{L}}{=}t^{q}\boldsymbol{Y},
\end{equation*}
where $\YY\sim \mu$ and $q$ is the self-similar exponent.
The time $t$ characteristic function of $\YY(t)\sim\mu_t$  is given by
\begin{equation*}
\hat{\mu}_t( \zz)=\hat{\mu}(  t^q\zz), \,\,\, \forall  \zz \in \RR^d.
\end{equation*}

\section{Factor-based $\rho\alpha$models}\label{Anig}

We first introduce the L\'evy  class of factor-based multivariate subordinators used
to construct the $\mathbb{R}^{n}$-valued asset return process $\{\boldsymbol{%
Y}(t),\,t\geq 0\}$. A multidimensional factor-based subordinator $\{%
\boldsymbol{G}(t),t\geq 0\}$ is defined as follows%
\begin{equation}
\boldsymbol{G}(t)=(X_{1}(t)+\alpha _{1}Z(t),...,X_{n}(t)+\alpha
_{n}Z(t)),\quad \alpha _{j}>0,\,j=1,...,n,  \label{subg}
\end{equation}%
where $\boldsymbol{X}(t)=\{(X_{1}(t),...,X_{n}(t)),t\geq 0\}$ and $%
\{Z(t),t\geq 0\}$ are independent subordinators with zero drift, and $%
\boldsymbol{X}(t)$ has independent components.
Let $\boldsymbol{B}(\boldsymbol{s})$
  and  $\boldsymbol{B}^{\rho }(t)=(B_{1}^{\rho
}(t),...,B_{n}^{\rho }(t))$ be the multivariate Brownian motions in Definition \ref{dipfb}, independent
of $\boldsymbol{B}(t)$.
The $\mathbb{R}^{n}$-valued subordinated process $\{\boldsymbol{Y}^L(t),t>0\}$
defined by%
\begin{equation}
\boldsymbol{Y}^L(t) =\left(
\begin{array}{c}
B_{1}(X_{1}(t))+B_{1}^{\rho }(Z(t)) \\
.... \\
B_{n}(X_{n}(t))+B_{n}^{\rho }(Z(t))%
\end{array}%
\right) ,  \label{abgp2}
\end{equation}%
where $X_{j}(t)$ and $Z(t)$ are independent subordinators, independent of $%
\boldsymbol{B}(t)$ and $\boldsymbol{B}^{\rho }(t)$ is a factor-based
subordinated Brownian motion, called $\rho\alpha$-model. Clearly if $\GG(1)$ has the same unit time distribution of   $\SS(1)$  in Definition  \ref{Satosub}, also $\YY^L(1)$ and $\YY^{\rho}(1)$ have the same distribution.

Let now $\boldsymbol{Y}^{L}(t)$ be the process  defined in \eqref{abgp2}. If $X(t)$ and $Z(t)$ are L\'evy subordinators with unit time distribution in \eqref{sub} and if we set $\mu_{j}=\beta_{j}\delta_{j}^{2}$ and $\sigma_{j}=\delta_{j}$ the process $\YY^{L}(t)$ is the $\rho\alpha$-NIG process in \cite{LuciSem1}.
Obviously $\YY^L(t)$ and $\YY^{\rho}(1)$ in Definition \ref{S-IG} have the same  distribution at unit time.

In \cite{marena2018multivariate} a Sato version of the $\rho\alpha$-NIG has been introduced.
 This process is termed
$\rho\alpha$-Sato NIG and it defined by:%

\begin{equation}
\boldsymbol{\XX}(t):\overset{\mathcal{L}}{=}t^{\boldsymbol{h}}\boldsymbol{Y}%
(1):=\left(
\begin{array}
[c]{c}%
t^{h_{1}}(B_{1}(X_{1}(1))+B_{1}^{\rho}(Z(1)))\\
....\\
t^{h_{n}}(B_{n}(X_{n}(1))+B_{n}^{\rho}(Z(1)))
\end{array}
\right)  , \label{Sabgp}%
\end{equation}
where $\boldsymbol{h}=(h_{1},\dots,h_{n})$ is the self-similar exponent. The unit time random vectors  $\XX(1)$, $\YY^L(1)$ and $\YY^{\rho}(1)$ have the same distribution.

 The process $\XX(t)$ is additive and  its distribution at unit time has one-dimensional normal inverse Gaussian marginals. Therefore it  has  one-dimensional marginal Sato processes (the normal inverse Gaussian distribution is self-decomposable).
In this work we consider the subcase  $h_1=\ldots=h_n=q$.
\bibliographystyle{apalike}
\bibliography{Biblio}

\begin{thebibliography}{}

\bibitem[Barndorff-Nielsen et~al., 2001]{Ba}
Barndorff-Nielsen, O.~E., Pedersen, J., and Sato, K.-i. (2001).
\newblock Multivariate subordination, self-decomposability and stability.
\newblock {\em Advances in Applied Probability}, 33(1):160--187.

\bibitem[Boen and Guillaume, 2019]{boen2019building}
Boen, L. and Guillaume, F. (2019).
\newblock Building multivariate {S}ato models with linear dependence.
\newblock {\em Quantitative Finance}, 19(4):619--645.

\bibitem[Buchmann et~al., 2015]{buchmann2015multivariate}
Buchmann, B., Kaehler, B., Maller, R., and Szimayer, A. (2015).
\newblock Multivariate subordination using generalised gamma convolutions with
  applications to vg processes and option pricing.
\newblock {\em arXiv preprint arXiv:1502.03901}.

\bibitem[Buchmann et~al., 2019]{buchmann2019weak}
Buchmann, B., Lu, K.~W., Madan, D.~B., et~al. (2019).
\newblock Weak subordination of multivariate {L}{\'e}vy processes and variance
  generalised gamma convolutions.
\newblock {\em Bernoulli}, 25(1):742--770.

\bibitem[Carr et~al., 2002]{carr2002fine}
Carr, P., Geman, H., Madan, D.~B., and Yor, M. (2002).
\newblock The fine structure of asset returns: An empirical investigation.
\newblock {\em The Journal of Business}, 75(2):305--332.

\bibitem[Carr et~al., 2007]{carr2007self}
Carr, P., Geman, H., Madan, D.~B., and Yor, M. (2007).
\newblock Self-decomposability and option pricing.
\newblock {\em Mathematical {F}inance}, 17(1):31--57.

\bibitem[Eberlein and Madan, 2009]{eberlein2009sato}
Eberlein, E. and Madan, D.~B. (2009).
\newblock Sato processes and the valuation of structured products.
\newblock {\em Quantitative Finance}, 9(1):27--42.

\bibitem[Folland, 2013]{folland2013real}
Folland, G.~B. (2013).
\newblock {\em Real analysis: modern techniques and their applications}.
\newblock John Wiley \& Sons.

\bibitem[Grigelionis, 2008]{grigelionis2008thorin}
Grigelionis, B. (2008).
\newblock Thorin classes of {L}{\'e}vy processes and their transforms.
\newblock {\em Lithuanian Mathematical Journal}, 48(3):294--315.

\bibitem[Guillaume, 2012]{guillaume2012sato}
Guillaume, F. (2012).
\newblock Sato two-factor models for multivariate option pricing.
\newblock {\em Journal of Computational Finance}, 15(4):159.

\bibitem[Guillaume, 2013]{guillaume2013alphavg}
Guillaume, F. (2013).
\newblock The $\alpha${VG} model for multivariate asset pricing: calibration
  and extension.
\newblock {\em Review of Derivatives Research}, 16(1):25--52.

\bibitem[Jevti{\'c} et~al., 2019]{jevtic2019multivariate}
Jevti{\'c}, P., Marena, M., and Semeraro, P. (2019).
\newblock Multivariate marked {P}oisson processes and market related
  multidimensional information flows.
\newblock {\em International Journal of Theoretical and Applied Finance},
  22(02):1850058.

\bibitem[Kokholm and Nicolato, 2010]{kokholm2010sato}
Kokholm, T. and Nicolato, E. (2010).
\newblock Sato processes in default modelling.
\newblock {\em Applied Mathematical Finance}, 17(5):377--397.

\bibitem[K{\"u}chler and Tappe, 2008]{kuchler2008bilateral}
K{\"u}chler, U. and Tappe, S. (2008).
\newblock Bilateral gamma distributions and processes in financial mathematics.
\newblock {\em Stochastic Processes and their Applications}, 118(2):261--283.

\bibitem[K{\"u}chler and Tappe, 2013]{kuchler2013tempered}
K{\"u}chler, U. and Tappe, S. (2013).
\newblock Tempered stable distributions and processes.
\newblock {\em Stochastic Processes and their Applications},
  123(12):4256--4293.

\bibitem[Li et~al., 2016]{li2016additive}
Li, J., Li, L., and Mendoza-Arriaga, R. (2016).
\newblock Additive subordination and its applications in finance.
\newblock {\em Finance and Stochastics}, 20(3):589--634.

\bibitem[Luciano et~al., 2016]{luciano2016dependence}
Luciano, E., Marena, M., and Semeraro, P. (2016).
\newblock Dependence calibration and portfolio fit with factor-based
  subordinators.
\newblock {\em Quantitative Finance}, 16(7):1037--1052.

\bibitem[Luciano and Semeraro, 2010a]{luciano2010generalized}
Luciano, E. and Semeraro, P. (2010a).
\newblock A generalized normal mean-variance mixture for return processes in
  finance.
\newblock {\em International Journal of Theoretical and Applied Finance},
  13(03):415--440.

\bibitem[Luciano and Semeraro, 2010b]{LuciSem1}
Luciano, E. and Semeraro, P. (2010b).
\newblock Multivariate time changes for {L}{\'e}vy asset models:
  Characterization and calibration.
\newblock {\em Journal of Computational and Applied Mathematics},
  233(8):1937--1953.

\bibitem[Lundin et~al., 1998]{lundin1998correlation}
Lundin, M.~C., Dacorogna, M.~M., and M{\"u}ller, U.~A. (1998).
\newblock Correlation of high frequency financial time series.
\newblock {\em Available at SSRN 79848}.

\bibitem[Madan and Seneta, 1990]{MS}
Madan, D.~B. and Seneta, E. (1990).
\newblock The variance gamma (vg) model for share market returns.
\newblock {\em Journal of {B}usiness}, pages 511--524.

\bibitem[Maejima et~al., 2009]{maejima2009note}
Maejima, M., Nakahara, G., et~al. (2009).
\newblock A note on new classes of infinitely divisible distributions on
  $\mathbb{R}^d$.
\newblock {\em Electronic Communications in Probability}, 14:358--371.

\bibitem[Marena et~al., 2018]{marena2018multivariate}
Marena, M., Romeo, A., and Semeraro, P. (2018).
\newblock Multivariate factor-based processes with {S}ato margins.
\newblock {\em International Journal of Theoretical and Applied Finance},
  21(01):1850005.

\bibitem[Mendoza-Arriaga and Linetsky, 2016]{mendoza2016multivariate}
Mendoza-Arriaga, R. and Linetsky, V. (2016).
\newblock Multivariate subordination of {M}arkov processes with financial
  applications.
\newblock {\em Mathematical Finance}, 26(4):699--747.

\bibitem[P{\'e}rez-Abreu and Stelzer, 2014]{perez2014infinitely}
P{\'e}rez-Abreu, V. and Stelzer, R. (2014).
\newblock Infinitely divisible multivariate and matrix gamma distributions.
\newblock {\em Journal of Multivariate Analysis}, 130:155--175.

\bibitem[Rosinski, 1990]{rosinski1990series}
Rosinski, J. (1990).
\newblock On series representations of infinitely divisible random vectors.
\newblock {\em The Annals of Probability}, 18(1):405--430.

\bibitem[Rosi{\'n}ski, 2007]{rosinski2007tempering}
Rosi{\'n}ski, J. (2007).
\newblock Tempering stable processes.
\newblock {\em Stochastic {P}rocesses and their {A}pplications},
  117(6):677--707.

\bibitem[Sato, 1982]{sato1982absolute}
Sato, K.-I. (1982).
\newblock Absolute continuity of multivariate distributions of class l.
\newblock {\em Journal of Multivariate Analysis}, 12(1):89--94.

\bibitem[Sato, 1999]{Sa}
Sato, K.-I. (1999).
\newblock {\em {L}{\'e}vy processes and infinitely divisible distributions}.
\newblock Cambridge university press.

\bibitem[Semeraro, 2008]{Sem1}
Semeraro, P. (2008).
\newblock A multivariate variance gamma model for financial applications.
\newblock {\em International Journal of Theoretical and Applied Finance},
  11(01):1--18.

\bibitem[Semeraro, 2020]{semeraro2020note}
Semeraro, P. (2020).
\newblock A note on the multivariate generalized asymmetric {L}aplace motion.
\newblock {\em Communications in Statistics-Theory and Methods},
  49(10):2339--2355.

\bibitem[Sun et~al., 2017]{sun2017marshall}
Sun, Y., Mendoza-Arriaga, R., and Linetsky, V. (2017).
\newblock Marshall-olkin distributions, subordinators, efficient simulation,
  and applications to credit risk.
\newblock {\em Advances in Applied Probability}, pages 481--514.

\bibitem[Takano, 1989]{takano1989mixtures}
Takano, K. (1989).
\newblock On mixtures of the normal distribution by the generalized gamma
  convolutions.
\newblock {\em Bulletin of the Faculty of Science, Ibaraki University. Series
  A, Mathematics}, 21:29--41.

\bibitem[Teng et~al., 2016]{teng2016dynamic}
Teng, L., Ehrhardt, M., and G{\"u}nther, M. (2016).
\newblock The dynamic correlation model and its application to the {H}eston
  model.
\newblock In {\em Innovations in Derivatives Markets}, pages 437--449.
  Springer, Cham.

\bibitem[T{\'o}th and Kert{\'e}sz, 2006]{toth2006increasing}
T{\'o}th, B. and Kert{\'e}sz, J. (2006).
\newblock Increasing market efficiency: Evolution of cross-correlations of
  stock returns.
\newblock {\em Physica A: Statistical Mechanics and its Applications},
  360(2):505--515.

\end{thebibliography}

\end{document}